\documentclass[12pt,draftcls,onecolumn]{IEEEtran}

\IEEEoverridecommandlockouts                              

\usepackage{graphicx} 
\usepackage{amsmath} 
\usepackage{amssymb}  
\usepackage{amsthm}
\usepackage[usenames,dvipsnames]{color}
\usepackage[caption=false,font=footnotesize]{subfig}
\usepackage{hyperref}

\hypersetup{
    colorlinks=true,       
    linkcolor=blue,          
    citecolor=blue,        
    filecolor=blue,      
    urlcolor=blue           
}

\newcommand{\suchthat}{\mid}
\newcommand{\supp}{\text{\textsf{supp}}}
\newcommand{\rowsupp}{\text{\textsf{rowsupp}}}
\newcommand{\field}[1]{\mathbb{#1}}

\newcommand{\NN}{\field{N}}
\newcommand{\RR}{\field{R}}
\newcommand{\CC}{\field{C}}

\DeclareMathOperator*{\argmin}{argmin}

\theoremstyle{definition}
\newtheorem{defn}{Definition}

\theoremstyle{plain}
\newtheorem{prop}{Proposition}
\newtheorem{lem}{Lemma}
\newtheorem{thm}{Theorem}

\theoremstyle{remark}
\newtheorem*{rem}{Remark}

\title{\LARGE \bf Secure estimation and control for cyber-physical systems under adversarial attacks}

\author{Hamza Fawzi, Paulo Tabuada, Suhas Diggavi
\thanks{This work was supported by the NSF award 1136174. The authors are with the \mbox{Department of Electrical Engineering},
University of California at Los Angeles. \texttt{hamzafawzi@gmail.com}, \texttt{tabuada@ee.ucla.edu}, \texttt{suhas@ee.ucla.edu}}}

\begin{document}

\maketitle


\begin{abstract}
The vast majority of today's critical infrastructure is supported by numerous feedback control loops and an attack on these control loops can have disastrous consequences. This is a major concern since modern control systems are becoming large and decentralized and thus more vulnerable to attacks. This paper is concerned with the estimation and control of linear systems when some of the sensors or actuators are corrupted by an attacker.

In the first part we look at the estimation problem where we characterize the resilience of a system to attacks and study the possibility of increasing its resilience by a change of parameters. We then propose an efficient algorithm to estimate the state despite the attacks and we characterize its performance. Our approach is inspired from the areas of error-correction over the reals and compressed sensing.

In the second part we consider the problem of designing output-feedback controllers that stabilize the system despite attacks. We show that a principle of separation between estimation and control holds and that the design of resilient output feedback controllers can be reduced to the design of resilient state estimators.

\end{abstract}




\section{Introduction}

Today's large-scale control systems are present everywhere in order to
sustain the normal operation of many of the critical processes that we
rely on. Example of such systems include chemical processes, the power
grid, water distribution networks and many more.

In a typical control system one can identify different components
including the actuators, the sensors and the controllers. These
different components need to communicate with each other: for example
the sensors communicate their measurements to the controllers, the
controllers use this information to compute the control input, and the
control input is then sent to the actuators so that it can be
physically implemented. In order for this communication to take place,
a communication network is usually deployed across the plant to be
controlled. Although wired networks have been traditionally used for
this purpose, an increasing number of control systems now use wireless
networks since they are easier to deploy and to maintain. In addition,
these networks are sometimes connected to the corporate intranet, and
in some cases even to the Internet. Consequently, modern control
systems are becoming more open to the \emph{cyber}-world, and as such,
are more vulnerable to attacks that can cause faults and failures in
the \emph{physical} process even though launched in the
\emph{cyber}-domain. This realization led to the emergence of new
security challenges that are distinct from traditional cyber security
as highlighted in \cite{sastry2009workshop,adam2010workshop}.

Real-world attacks on control systems have in fact occurred in the
past decade and have in some cases caused significant damage to the
targeted physical processes. Perhaps one of the most popular examples
is the attack on Maroochy Shire Council's sewage control system in
Queensland, Australia that happened in January 2000
\cite{cardenas2008research,slay2007lessons}. In this incident, an
attacker managed to hack into some controllers that activate and
deactivate valves and, by doing so, caused flooding of the grounds of
a hotel, a park, and a river with a million liters of sewage
\cite{cardenas2008research}. Another well
  publicized example of an attack launched on physical systems is the
very recent StuxNet virus that targeted Siemens' supervisory control
and data acquisition systems which are used in many industrial
processes \cite{ComputerWorldStuxnet}. Other cases of attacks have
been reported in the past years, and we refer the reader to
\cite{cardenas2008research} for more real-world examples.

These examples indicate the clear need for strategies and mechanisms
to identify and deal with attacks on control systems.

\noindent \textbf{Previous work related to security for control
  systems.} The design of control and estimation algorithms that are
resilient against faults and failures is certainly not a new
problem. In fault-detection and identification
\cite{massoumnia1989failure,blanke2006diagnosis} the objective is to
detect if one or more of the components of a system has
failed. Traditionally, this is done by comparing the measurements of
the sensors with an analytical model of the system and by forming the
so-called residual signal (in some cases, the residual signal actually
corresponds to the output of some specifically designed LTI system
whose inputs are the sensor measurements
\cite{massoumnia1989failure}). This residual signal is then analyzed
(e.g., using signal processing techniques) in order to determine if a
fault has occurred. In such algorithms however, there is in general
one residual signal \emph{per failure mode}. As we will see later, in
our problem formulation the number of failure modes can be very large
and one cannot afford to generate and analyze a residual signal for
each possible failure mode. In another area, namely robust control
\cite{zhou1998essentials}, one seeks to design control methods that
are robust against disturbances in the model. In general however,
these disturbances are treated as \emph{natural} disturbances to the
system and are assumed to be \emph{bounded}. This does not apply in
the context of security since the disturbances will typically be
adversarial and therefore cannot be assumed bounded. This is the case
also in the area of stochastic control and estimation, where the
disturbances are assumed to follow a certain probabilistic model,
which we cannot adopt for our problem.

Since these assumptions are not justifiable in the context of
adversarial attacks, there has been a recent increase in control
systems security research
\cite{schenato2007foundations,gupta2010optimal,
  pasqualetti2012attackI, pasqualetti2012attackII,
  sundaram2010wireless,sundaram2008distributed,teixeira2010cyber}. In~\cite{schenato2007foundations},
the authors consider the problem of control and estimation in a
networked system when the communication links are subject to
disturbances. The disturbances (corresponding to packet losses) are
however assumed to follow a certain stochastic process (typically a
Bernoulli process) which does not necessarily capture the behavior of
an attacker. In \cite{gupta2010optimal} the authors consider a more
intelligent jammer who plans his attacks in order to maximize a
certain cost, while the objective of the controller is to
minimize this same cost. The authors showed the existence
of saddle-point equilibrium for this dynamic zero-sum game and derived
the optimal jamming strategy for a particular instance of the
problem. The results are however derived in the case of
one-dimensional systems only, which is a main limitation of this work.
In \cite{pasqualetti2012attackI, pasqualetti2012attackII} the authors study the fundamental limitations of attack detection and identification methods for linear systems, and for the particular case of power networks. They provide graph-theoretic characterization of the vulnerability of such systems to attacks, and furthermore they propose centralized and decentralized filters to detect and identify attacks when possible. These filters are however computationally expensive and are in general difficult to implement.
Another related problem that received attention recently is the problem of reaching consensus in the presence of malicious agents \cite{sundaram2008distributed}. The authors characterize the number of infected nodes that can be tolerated and propose a way to overcome the effect of the malicious agents when possible. However one particularity of these works is that the dynamics is part of the algorithm and can be specifically designed, rather than being given as in a physical system.
Finally, there has also been recent work in the area of real error-correction over adversarial channels, e.g., \cite{candes2005decoding}, where adversarial noise could be unbounded. However the dynamics of the system does not generally play a role and the correction capability is studied in a static setting that does not take advantage of the dynamics of the system. Furthermore, in these works also, the error protection mechanism can be designed by suitably choosing the coding matrix, whereas in our case, the plant dynamics is given to us.

\noindent \textbf{Contributions and organization of the paper.} In
this work we adopt a novel point of view inspired from
error-correction over the reals \cite{candes2005decoding} which allows
us to propose a new estimation algorithm that is robust against the
attacks and that is also computationally efficient, unlike most of the
previously proposed approaches. Furthermore, in
contrast with some of the previously described work, we do not
restrict the type of attacks introduced by the attacker on the
captured nodes (in particular the attacks injected can be of arbitrary
magnitude). In our framework, the attacks are modeled as sparse
vectors that affect the outputs (sensor attacks) as well as the inputs
(actuator attacks).

The contributions of this paper can be divided into two main parts:

\begin{enumerate}
\item The first part (section \ref{sec:Estimation}) deals with the
  \emph{estimation} problem in the presence of sensors and actuator
  attacks. We first characterize the resilience of a system and the
  maximum number of attacked nodes that can be tolerated for correct
  estimation. We then propose a computationally feasible decoding
  algorithm to recover the state despite the attacks. This algorithm
  is inspired from the area of compressed sensing and its relation to
  error-correction over the reals \cite{candes2005decoding}. Finally
  we show that if we can implement a state-feedback law (i.e., change
  the dynamics matrix $A$ to be $A+BK$), then one can always increase
  the resilience of a system while still having freedom in choosing
  the performance (i.e., the eigenvalues) of the system.  This first
  part of the paper mainly focuses on attacks on sensors for ease of
  exposition, but we also show at the end of the section how the
  decoder and some of the results can be extended to the case of
  attacks on actuators.

\item The second part of the work (section \ref{sec:Control}) deals
  with the problem of \emph{control with output feedback} in the
  presence of attacks on sensors. There we consider the question of
  designing an output-feedback law that stabilizes the system despite
  attacks on sensors. Our main result in this section is to show that
  if such a stabilizing law exists, then the state can also be
  estimated despite the attacks on sensors. This means that the
  estimation and the stabilization problems in the presence of sensor
  attacks are in some sense equivalent. Hence, when designing an
  output-feedback stabilization law, one can instead focus on the
  estimation problem, and use a state estimator resilient to attacks
  with any standard \emph{state}-feedback law to obtain a stabilizing
  output-feedback law resilient to attacks (separation of estimation
  and control).
\end{enumerate}

Preliminary versions of the results appeared in the conference papers
\cite{fawzi2011secure,fawzi2012security} as well as in the Master's
thesis of the first author \cite{fawzi2011thesis}.

\section{The formal setting and notations}
\label{sec:ProblemFormulation}

\noindent \textbf{The formal setting.}
Consider the linear control system given by the equations:
\begin{equation}
\begin{aligned}
x^{(t+1)} &= Ax^{(t)}+B(K^{(t)}(y^{(0)}, \dots, y^{(t)}) + w^{(t)})\\
y^{(t)}   &= Cx^{(t)} + e^{(t)}
\end{aligned}
\label{eq:control_system_with_errors}
\end{equation}

Here $x^{(t)} \in \RR^n$ represents the state of the system at time $t
\in \mathbb{N}$, and $y^{(t)} \in \RR^p$ is the output of the sensors
at time $t$. The control input applied at time $t$ depends on the past
measurements $(y^{(\tau)})_{0 \leq \tau \leq t}$ through the output
feedback map $K^{(t)}$. The vector $e^{(t)} \in \RR^p$ represents the
attacks injected by the attacker in the different sensors, and the
vector $w^{(t)} \in \RR^m$ represents the attacks injected in the
actuators. Note that if sensor $i \in \{1,\dots,p\}$ is not attacked
then necessarily $e^{(t)}_i = 0$ and the output $y^{(t)}_i$ of sensor
$i$ is not corrupted, otherwise $e^{(t)}_i$ (and therefore
$y^{(t)}_i$) can take any value. The sparsity pattern of the attack
$e^{(t)}$ therefore indicates the set of attacked sensors. The same
observation holds for the attacks on actuators $w^{(t)}$.

Note that from a practical point of view, an attack on a sensor could
either be interpreted as an attack on the node itself (making it
transmit an incorrect signal), or it could also be interpreted as an
attack on the \emph{communication link} between the sensor and the
receiver device. Similarly an attack on an actuator could either be
interpreted as an attack on the actuator itself, or on the
communication link from the controller to the actuator. Throughout the
paper, we will be talking about ``attacked nodes'' but we will keep in
mind that the second interpretation (attack on the communication link)
is also possible. In section
\ref{sec:IncreasingNumberCorrectableErrors} we will actually look at
a scenario where it is the communication links that are compromised
and not the nodes themselves.

We will assume in this paper that the set of attacked nodes does not
change over time. More precisely, if $K \subset \{1,\dots,p\}$ is the
set of attacked sensors and $L \subset \{1,\dots,m\}$ the set of attacked actuators, then we have for all $t$,
$\supp(e^{(t)}) \subset K$ and $\supp(w^{(t)}) \subset L$ (where $\supp(x)$ denotes the support of $x$, i.e., the indices of the nonzero components of $x$). Note that this is a valid and
realistic assumption when the time it takes for the malicious agent to
gain control of a node is large compared to the time scale of the
estimation algorithm. Furthermore observe that a model where the set of attacked
nodes is allowed to change at every time step while having a fixed
cardinality would in turn not be very realistic since it would assume
that the attacker abandons from $t$ to $t+1$ some of the nodes he
had control over. For these reasons, we will assume for our model that
the sets $K$ and $L$ of attacked sensors and actuators is constant over time (and, of course,
unknown).

Moreover, since we are dealing with
a malicious agent, we will not assume the attacks $e^{(t)}_i$ or $w^{(t)}_j$ (for an
attacked sensor $i$ or actuator $j$) to follow any particular model and we will simply
take them to be arbitrary real numbers. The only assumption concerning
the malicious agent will be about the \emph{number} of nodes that
were attacked. Our statements will then typically characterize the
number of attacks that can be tolerated in order to correctly estimate/control the plant.

\noindent \textbf{Notations.} We use the following notations throughout the paper. If $S$ is a set, we denote by $|S|$ the cardinality of $S$ and by $S^c$ the complement of $S$. For a vector $x \in \field{R}^n$, the support of $x$, denoted by $\supp(x)$, is the set of nonzero components of $x$ and the $\ell_0$ norm of $x$ is the number of nonzero components of $x$:
\[\supp(x) = \{i \in \{1, \dots, n\} \suchthat x_i \neq 0\}, \qquad \|x\|_{\ell_0} = |\supp(x)|.\]
Also, if $K \subset \{1,\dots,n\}$, we let $\mathcal P_K$ be the
projection map onto the components of $K$ ($\mathcal P_K x$ is a
vector with $|K|$ components, e.g., if $K=\{3,5\}$ then $\mathcal{P}_Kx=(x_3,x_5)$).

For a matrix $M \in \field{R}^{m\times n}$ we denote by $M_i \in
\field{R}^n$ the $i$'th row of $M$, for $i \in \{1, \dots, m\}$. We
define the \emph{row support} of $M$ to be the set of nonzero rows of
$M$ and we denote by $\|M\|_{\ell_0}$ the cardinality of the row support of $M$:
\[ \rowsupp(M) = \{i \in \{1, \dots, m\} \suchthat M_i \neq 0 \}, \qquad \|M\|_{\ell_0} = |\rowsupp(M)|. \]

\section{The estimation problem}
\label{sec:Estimation}

In this section we deal with the problem of \emph{estimating} the state of a linear dynamical system in the presence of attacks. Throughout the main part of this section we will assume that attacks only occur on the sensors (i.e., no attacks on actuators) for ease of exposition. At the end of the section though we show how to extend the results to the case where there are also attacks on the actuators.

We consider in this section linear dynamical systems of the form:
\begin{equation}
\begin{aligned}
x^{(t+1)} &= A x^{(t)}\\
y^{(t)} &= C x^{(t)} + e^{(t)}
\end{aligned}
\label{eq:model}
\end{equation}
As mentioned before, $e^{(t)} \in \field{R}^p$ are the attack vectors injected by the malicious
agent in the sensors. For simplicity we have also discarded the control input $BK^{(t)}(y^{(0)},\dots,y^{(t)})$ since it does not affect the results in this section. Indeed the results presented here hold for any linear affine system where the state evolves according to $x^{(t+1)} = Ax^{(t)} + v^{(t)}$ where $v^{(t)}$ is a \emph{known} input (for more details on this, see section \ref{sec:Control_Properties}).

The problem that we consider in this section is to reconstruct the initial state $x^{(0)}$ of the plant from the corrupted observations $(y^{(t)})_{t=0,\dots,T-1}$. Note that since the matrix $A$ is known, the problem of reconstructing the current state $x^{(t)}$ or the initial state $x^{(0)}$ are --at least theoretically-- equivalent. Therefore, there is no loss of generality in focusing on the reconstruction of $x^{(0)}$ instead of the current state $x^{(T-1)}$.

\subsection{Error correction and number of correctable attacks}
\label{sec:errorCorrection}

Let $x^{(0)} \in \field{R}^n$ be the initial state of the plant and let $y^{(0)}, \dots, y^{(T-1)} \in \field{R}^p$ be the first $T$ measurements that are transmitted from the sensors to the receiver device. The objective of the receiver device is to reconstruct the initial state $x^{(0)}$ from these measurements. These vectors are given by
\[ y^{(t)} = CA^t x^{(0)} + e^{(t)}, \]
 where $e^{(t)}$ represent the
error vector (i.e., the attack vector) injected by the attacker (throughout the paper we will use the terms ``error vector'' and ``attack vector'' interchangeably to designate the vector $e^{(t)}$ injected by the attacker; the term ``error vector'' emphasizes the error-correction perspective we adopt in this paper). Recall that $\supp(e^{(t)})
\subset K$ with $K \subset \{1,\dots,p\}$ being the set of sensors
that are attacked and whose data is unreliable.

Having received the $T$ vectors $y^{(0)}, \dots, y^{(T-1)}$, the
receiver uses a decoder $D\colon\left(\field{R}^{p}\right)^T
\rightarrow \field{R}^n$ in order to estimate the initial state
$x^{(0)}$ of the plant. The decoder correctly estimates the initial
state if $D(y^{(0)}, \dots, y^{(T-1)}) = x^{(0)}$.

We say that the decoder $D$ corrects $q$ errors if it correctly recovers the initial state $x^{(0)}$ for any set $K$ of attacked sensors of cardinality less than or equal to $q$. More formally we introduce the following definition:

\begin{defn}
We say that $q$ errors are correctable after $T$ steps by the decoder $D\colon\left(\field{R}^{p}\right)^T \rightarrow
\field{R}^n$ if for any $x^{(0)} \in \field{R}^n$, and for any
sequence of vectors $e^{(0)}, \dots, e^{(T-1)}$ in $\field{R}^p$ such
that $\supp(e^{(t)}) \subset K$ with $|K| \leq q$, we have
$D(y^{(0)},\dots,y^{(T-1)}) = x^{(0)}$ where $y^{(t)} = CA^t x^{(0)} +
e^{(t)}$, $t=0,\hdots,T-1.$\\
Furthermore, we say that $q$ errors are correctable after $T$ steps (or, equivalently, that the system is \emph{resilient} against $q$ attacks after $T$ steps) if there exists a decoder that can correct $q$ errors after $T$ steps.
\label{defn:qerrcorr}
\end{defn}

Let $E_{q,T}$ denote the set of error vectors $(e^{(0)}, \dots, e^{(T-1)}) \in (\field{R}^p)^T$ that satisfy $\forall t \in \{0,\dots,T-1\}, \; \supp(e^{(t)}) \subset K$ for some $K \subset \{1,\dots,p\}$ with $|K| \leq q$. Note that $E_{q,T}$ is a union of $\binom{p}{q}$ subspaces in $(\field{R}^p)^T$.

\subsubsection{Characterization of the number of correctable errors}

Observe that, by definition~\ref{defn:qerrcorr}, the existence of a decoder that can correct $q$ errors is equivalent to saying that the following map
\begin{equation}
\begin{aligned}
    \field{R}^n \times E_{q,T} & \; \rightarrow \; (\field{R}^p)^T \\
    (x^{(0)}, e^{(0)}, \dots, e^{(T-1)})  &\; \mapsto \; (y^{(0)}, \dots, y^{(T-1)}) = (Cx^{(0)} + e^{(0)}, \dots, CA^{T-1} x^{(0)} + e^{(T-1)})
\end{aligned}
\label{eq:mapinj}
\end{equation}
is invertible, or, more precisely, that it has an inverse for the
first $n$ components of its domain (we are only interested in the state $x^{(0)}$, and not necessarily the error vectors). \footnote{These two conditions --the existence of an inverse and the existence of an inverse to recover just the first $n$ components-- are actually equivalent since the attack vectors are uniquely determined by
$x^{(0)}$ and the $y^{(t)}$'s and are given by $e^{(t)} = y^{(t)} -
CA^t x^{(0)}$.}  Thus expressing injectivity of this map is equivalent
to saying that $q$ errors are correctable. This gives the following
proposition:

\begin{prop}
Let $T \in \field{N}\backslash\{0\}$. The following are equivalent:\\
(i) There is \emph{no} decoder that can correct $q$ errors after $T$ steps;\\
(ii) There exists $x_a, x_b \in \field{R}^n$ with $x_a \neq x_b$, and error vectors
\mbox{$(e_a^{(0)}, \dots, e_a^{(T-1)}) \in E_{q,T}$} and \mbox{$(e_b^{(0)}, \dots,
e_b^{(T-1)}) \in E_{q,T}$} such that $A^t x_a + e_a^{(t)} = A^t x_b +
e_b^{(t)}$ for all \mbox{$t \in \{0,\dots,T-1\}$}.
\label{prop:injectivity}
\end{prop}

The proposition above simply says that it is not possible to unambiguously recover the state $x^{(0)}$ if there are two distinct values $x_a$ and $x_b$ with $x_a \neq x_b$ that can, with less than $q$ corrupted sensors, explain the received data.

Note that the domain of the map defined in \eqref{eq:mapinj} is the
Cartesian product of the whole $\field{R}^n$ with the error set
$E_{q,T}$ which is unbounded.  This means that we require the decoder
to recover \emph{any} initial state $x^{(0)}$ for \emph{any} sequence
of error vectors from $E_{q,T}$. In practice however one could
consider only vectors $x^{(0)}$ in some set $\Omega \subset
\field{R}^n$ if one has prior knowledge on the initial state (for
example, if the states are all nonnegative, say for physical reasons,
then one could take $\Omega=\field{R}^n_{+}$). Similarly, if the
attacker has a finite amount of energy then we could envisage
considering only elements of $E_{q,T}$ in a certain ball of finite
radius. We do not however pursue this here, and we assume in
particular that the initial state of the plant can be anywhere in
$\field{R}^n$ and that the magnitude of the errors can be arbitrary.

We now give a necessary and sufficient condition for $q$ errors to be correctable that is simpler than the one in proposition \ref{prop:injectivity}.

\begin{prop}
Let $T \in \field{N}\backslash\{0\}$. The following are equivalent:\\
(i) There is a decoder that can correct $q$ errors after $T$ steps;\\
(ii) For all $z \in \field{R}^n \backslash \{0\}$, $|\supp(Cz) \cup \supp(CAz) \cup \dots \cup \supp(CA^{T-1} z)| > 2q$.
\label{prop:NSCl0}
\end{prop}
\begin{proof}
\textbf{(i)$\Rightarrow$ (ii)}: Suppose for the sake of contradiction
that there exists $z \in \field{R}^n \backslash \{0\}$ such that
$|\supp(Cz) \cup \supp(CAz) \cup \dots \cup \supp(CA^{T-1} z)| \leq
2q$. Let $L_a$ and $L_b$ be two disjoint subsets of $\{1,\dots,p\}$ with $|L_a|\leq q$ and $|L_b|\leq q$ such that $L_a\cup L_b = \supp(Cz) \cup \dots \cup \supp(CA^{T-1}z)$ (such $L_a$ and $L_b$ exist since $|\supp(Cz) \cup \supp(CAz) \cup \dots \cup \supp(CA^{T-1} z)| \leq
2q$). Let $e^{(t)}_a  = CA^t z|_{L_a}$ be the vector obtained from $CA^t z$ by setting all the components outside of $L_a$ to 0, and similarly let $e^{(t)}_b = - CA^t z|_{L_b}$. Then we have $CA^t z = e^{(t)}_a
- e^{(t)}_b$ with $\supp(e^{(t)}_a) \subset L_a$ and $\supp(e^{(t)}_b)
\subset L_b$ with $|L_a| \leq q$ and $|L_b| \leq q$. Now let, for $t \in \{0,\dots,T-1\}$, $y^{(t)} =
CA^t z + e^{(t)}_b = CA^t \cdot 0 + e^{(t)}_a$. If $q$ errors were
correctable after $T$ steps by some decoder $D$ then we would have $D(y^{(0)}, \dots,
y^{(T-1)}) = z$ and also $D(y^{(0)}, \dots, y^{(T-1)}) = 0$ which is
impossible since $z\neq 0$.

\textbf{(ii)$\Rightarrow$ (i)}: We again resort to
contradiction. Suppose that $q$ errors are \emph{not} correctable
after $T$ steps: this means there exists $x_a \neq x_b$, and error
vectors $e_a^{(0)}, \dots, e_a^{(T-1)}$ (supported on $L_a$ with
$|L_a| \leq q$) and $e_b^{(0)}, \dots, e_b^{(T-1)}$ (supported on
$L_b$, with $|L_b| \leq q$) such that $CA^t x_a + e_a^{(t)} = CA^t x_b
+ e_b^{(t)}$ for all $t \in \{0,\dots,T-1\}$. Now let $z = x_a - x_b
\neq 0$. If we let $L = L_a \cup L_b$, then we have $|L| \leq 2q$, and
we have for all $t \in \{0,\dots,T-1\}$, $\supp(CA^t z) \subset L$
which shows that (ii) does not hold.
\end{proof}

It is interesting to note the connection of the proposition above with
the definition of a $q$-error-correcting linear code in the context of
coding over the real numbers. A matrix $C \in \field{R}^{p \times n}$
(with $p>n$) defines a $q$-error-correcting linear
code if for any $z \neq 0$, $|\supp(Cz)| > 2q$ (see for example
\cite[\textsection 3]{guruswami2008euclidean}). This is precisely the
condition we obtain from the previous proposition when $T=1$ or when
there is no dynamics.

It is also interesting to observe that the proposition above
shows that one cannot recover the initial state $x^{(0)}$ until the
observability matrix given by
\[ \begin{bmatrix} C \\ CA \\ \vdots \\ CA^{T-1} \end{bmatrix} \]
has rank $n$. Indeed, if the observability matrix has rank smaller
than $n$ then it has a nontrivial kernel and there exists $z \neq 0$
such that $Cz=CAz=\dots=CA^{T-1}z=0$. This shows, by the above
proposition, that ``0 errors cannot be corrected'', or in other words,
that one cannot reconstruct $x^{(0)}$ even if there are no errors in
the $y^{(0)}, \dots, y^{(T-1)}$. The condition stated in proposition
\ref{prop:NSCl0} can therefore be seen as a generalized condition for
observability of a linear dynamical system when the observations are
corrupted (as per the model considered here).

Observe also that the characterization of proposition \ref{prop:NSCl0}
shows that the maximum number of correctable errors cannot increase
beyond $T=n$ measurements. Indeed, this is a direct consequence of the
Cayley-Hamilton theorem since we have for any $z$ and for $t \geq n$,
$\supp(CA^{t} z) \subset
\supp(Cz)\cup\supp(CAz)\cup\dots\cup\supp(CA^{n-1}z)$.

Finally, one can also directly see from the same proposition that the
number of correctable errors is always less than $p/2$, for any
$T$.
It turns out actually that \emph{generically} (i.e., for ``almost all'' systems $(A,C)$), the number of correctable errors is maximal and equal to $\lceil p/2 - 1 \rceil$.

\begin{prop}
  For almost all\footnote{That is, except on a set of Lebesgue measure
    zero} pairs $(A,C) \in \field{R}^{n \times n} \times \field{R}^{p
    \times n}$ the number of correctable errors after $T=n$ steps is
  maximal and equal to $\lceil p/2 - 1 \rceil$.
\end{prop}
\begin{proof}
We defer the proof of this proposition to section \ref{sec:Estimation_ActuatorAttacks} where we prove a more general result that takes into account attacks on actuators (cf. proposition \ref{prop:genericResilienceActuators}).
\end{proof}

\subsubsection{Computing the number of correctable errors}
\label{sec:numcorrl0}

Even though for \emph{almost all} pairs
$(A,C)$ the number of errors that can be corrected is maximal
(equal to $\lceil p/2 - 1 \rceil$ for $T=n$), the problem of
actually \emph{computing} the number of errors that can be corrected for a
given pair $(A,C)$ after a given number of steps $T$ is a hard problem in general. Actually one simple yet expensive algorithm is to look for the smallest $|K|$ where $K \subset \{1,\dots,p\}$ for which the following matrix has a nontrivial kernel:
\[ \begin{bmatrix}
\mathcal P_{K^c} C\\
\mathcal P_{K^c} CA\\
\vdots\\
\mathcal P_{K^c} CA^{T-1}
\end{bmatrix}. \]
If $s$ is the cardinality of the smallest $K$ for which this matrix has nontrivial kernel, then by proposition \ref{prop:NSCl0} the maximum number of correctable errors is $\lceil s/2 - 1\rceil$. This algorithm is however computationally expensive and requires
computing the rank of $2^p$ matrices in the worst-case. A recent result \cite{tillmann2012complexitySpark} shows that it is very unlikely that there is a more efficient way to perform the computation
\footnote{In the special case $T=1$ of error
  correction without dynamics, the number of errors that can be
  corrected is directly related to the \emph{spark} of a matrix $F$ that annihilates 
  $C$, i.e., such that $FC=0$ (see \cite[\textsection I.G]{candes2005decoding}). The spark of
  a matrix $F$ is the smallest number of columns that are linearly
  dependent. According to the recent paper \cite{tillmann2012complexitySpark}, computing the spark of a matrix is NP-hard.}.


\subsection{Increasing the number of correctable errors by state feedback}
\label{sec:IncreasingNumberCorrectableErrors}

In this section we consider the question of whether it is possible to make a given system $(A,C)$ more resilient against attacks by modifying the parameters of the system. More specifically, if $B$ is some given matrix, we look at the problem of designing a matrix $K$ so that the pair $(A+BK,C)$ is resilient against a large number of attacks, while it satisfies at the same time other design constraints.

\begin{figure}[htbp]
\centering
\includegraphics[width=7.5cm]{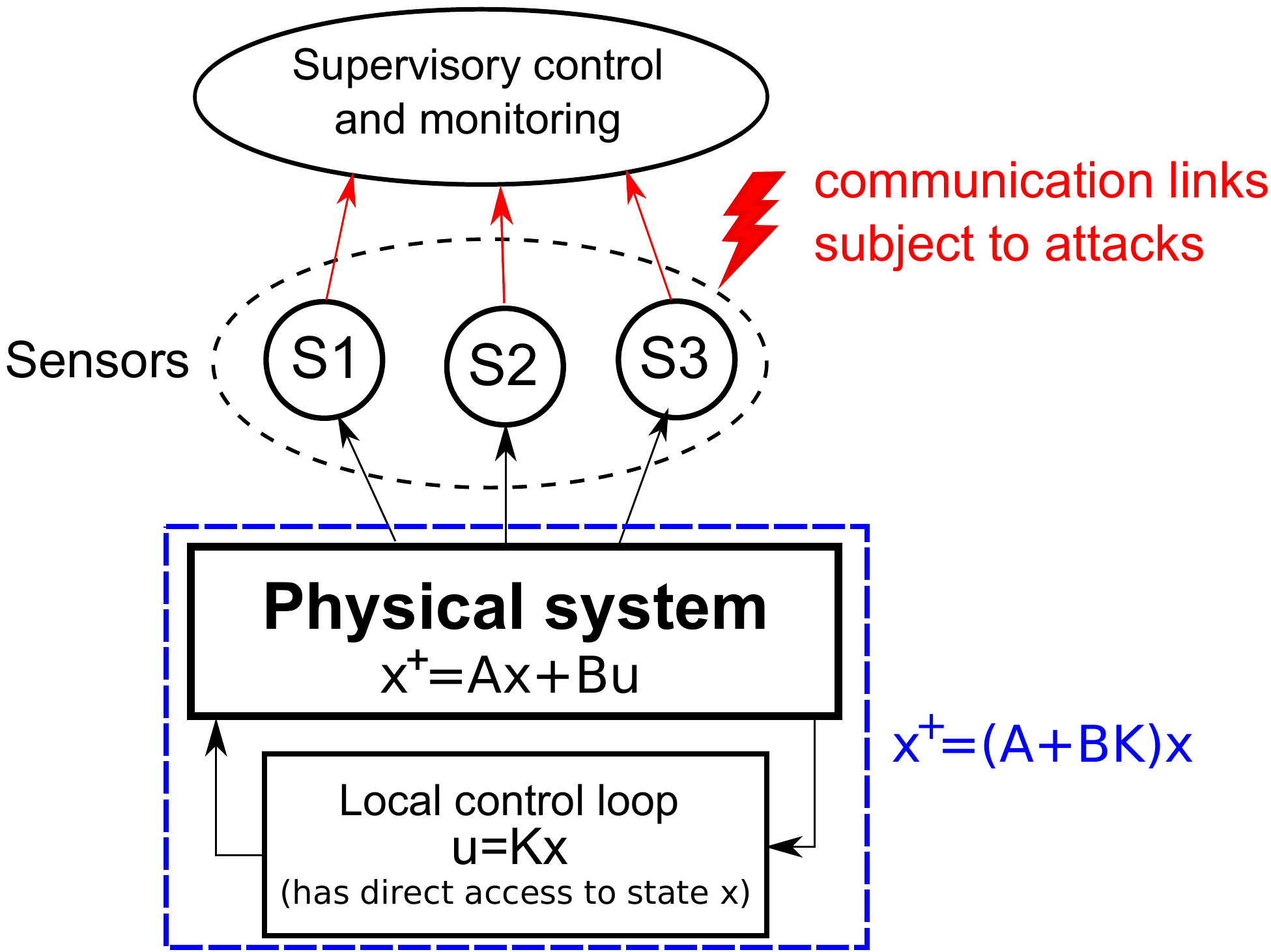}
\caption{Scenario where a local control loop has direct access to uncorrupted sensor information. Using this local control loop, the evolution of the physical system will be governed by the matrix $A+BK$ where $A$ is the open-loop matrix, $B$ is the control matrix, and $K$ can be chosen arbitrarily. The objective is to find $K$ such that the pair $(A+BK,C)$ is resilient against a large number of attacks. Choosing such a $K$ will allow the higher level supervisory control and monitoring system to recover the correct state despite attacks in the communication links between the sensors and the supervisory system.}
\label{fig:LocalControlLoop}
\end{figure}

From a practical point of view, this question can be motivated by the
following scenario depicted in Figure \ref{fig:LocalControlLoop}: we
first assume that the physical system possesses a local control loop
that has direct access to the state of the plant and that can control
the evolution of the physical system. This is possible for example if
the sensors are connected to the local controller through a wired link
that is not subject to external attacks. If the local control loop
implements a feedback law of the form $u=Kx$ then the evolution of the
physical system is governed by the matrix $A+BK$. Also, and as part of
the overall plant, a high-level supervisory and monitoring system
receives measurements from the sensors through wireless and vulnerable
communication links that are subject to attacks. Observe that the
choice $K$ of the local controller will affect the resilience of the
system to attacks, i.e., how many errors are correctable by the
supervisory system. The objective here is therefore to design $K$ in
order to make the number of correctable errors of the pair $(A+BK,C)$
as large as possible.

Note that there are other design constraints that come into play in
the choice of the local feedback law. Typically $K$ is chosen so that
the eigenvalues of $A+BK$ are inside the unit disc so that the
resulting closed-loop system is stable. It is known by the pole
placement theorem that this is possible if the pair $(A,B)$ is
controllable \cite{antsaklis2005linear}.

In this section we ask if one can also enforce the
  requirement that the number of correctable errors of the new pair
$(A+BK,C)$ is large, without losing the freedom of choosing the
eigenvalues of $A+BK$. We show in this section that the answer is yes,
and that if the pair $(A,B)$ is controllable, then it is possible to
choose $K$ such that $\lceil p/2 - 1 \rceil$ errors are correctable
for $(A+BK,C)$ and such that the eigenvalues of $A+BK$ are in any
arbitrary (or almost arbitrary) prescribed locations in the complex
plane. In other words, by an adequate choice of the local control law,
one can make the system \emph{more resilient to attacks} (the number
of correctable errors $\lceil p/2 - 1 \rceil$ is the maximum
possible), without compromising the control performance.

More specifically, we have the following result\footnote{We deal only
  with the single-input case here but the multi-input case can be
  handled using the same arguments. Moreover, the
    condition for $\{\lambda_i\}$ to have distinct amplitudes is not
    much of a restriction since one can always choose the $\lambda_i$'s to satisfy this condition and the consequences in terms of performance are negligible.}:
\begin{prop}
\label{prop:stateFeedback}
Let $A \in \field{R}^{n \times n}$, $B \in \field{R}^{n \times 1}$ and
$C \in \field{R}^{p \times n}$ and assume that the pair $(A,B)$ is
controllable.  Then there exists a finite set $F \subset \CC$ such
that for any choice of $n$ numbers $\lambda_1, \dots, \lambda_n \in
\field{C}\backslash F$ such that the $\lambda_i$'s have distinct
magnitudes, there exists $K \in \field{R}^{1 \times n}$ such that:
\begin{itemize}
\item the eigenvalues of the closed-loop matrix $A+BK$ are $\lambda_1,
  \dots, \lambda_n$.
\item the number of correctable errors after $n$ steps for the pair
  $(A+BK,C)$ is maximal (equal to $\lceil p/2 - 1 \rceil$).
\end{itemize}
\end{prop}

In order to prove this result, we make use of the following lemma:
\begin{lem}
\label{lem:CorrectionEigenvector}
Let $A \in \field{R}^{n \times n}$ and $C \in \field{R}^{p \times n}$. Assume that $A$ has $n$ eigenvalues all with distinct magnitudes (in particular $A$ is diagonalizable). Then the following are equivalent:\\
(i) $q$ errors are correctable for $(A,C)$ after $n$ steps.\\
(ii) for every eigenvector $v$ of $A$, $|\supp(Cv)| > 2q$.
\end{lem}
\begin{proof}
\begin{itemize}
\item (i) $\Rightarrow$ (ii): This direction simply corresponds to taking $x$ to be an eigenvector of $A$ in the condition $|\supp(Cx)\cup\dots\cup\supp(CA^{n-1} x)| > 2q$ of Proposition \ref{prop:NSCl0}.
\item (ii) $\Rightarrow$ (i): We assume that all eigenvectors $v$ of $A$ satisfy $|\supp(Cv)| > 2q$ and we will show that for any $x \neq 0$, we have $|\supp(Cx) \cup \supp(CAx) \cup \dots \cup\supp(CA^{n-1} x)| > 2q$. The idea here is that if $x \neq 0$ then for $t$ large enough the vector $A^t x$ will be very close to an eigenvector $w$ of $A$, and hence the support of $CA^t x$ will have more than $2q$ elements since $|\supp(Cw)| > 2q$. More formally, let $x \in \RR^n \backslash \{0\}$ and consider the decomposition of $x$ in the eigenbasis of $A$: $x = \sum_{i=1}^s \alpha_i v_i$ with $\alpha_i \neq 0$ for at least one $i$, and where $v_1, \dots, v_s$ are eigenvectors of $A$ associated with eigenvalues $\lambda_1, \dots, \lambda_s$. Since the eigenvalues of $A$ have distinct magnitudes we can assume that $|\lambda_1| > |\lambda_2| > \dots > |\lambda_s|$. We isolate the largest eigenvalue in this decomposition and we denote $\lambda = \lambda_1$ and $w = v_1$. Now we have $(A^t x - \alpha_1 \lambda^t w)/\lambda^t \rightarrow 0$ when $t \rightarrow +\infty$. Let $S = \supp(Cw)$. Note that since $w$ is an eigenvector of $A$ we have $|S| > 2q$ (by assumption). We'll now show that for $t$ large enough, the support of $CA^t x$ contains $S$: let $\beta = \min_{i \in S} |(Cw)_i|$ and observe that clearly $\beta > 0$. Let $t$ be large enough so that $\frac{|C(A^t x - \alpha_1 \lambda^t w)|_i}{|\lambda|^t} < \beta/2$ for all $i \in S$. Now we have, for $i \in S$:
\[ \frac{1}{|\lambda|^t} |CA^t x|_i \geq \frac{1}{|\lambda|^t} (|\lambda^t Cw|_i - |CA^t x - \lambda^t Cw|_i) > \beta - \beta/2 = \beta/2 > 0. \]
Hence $S \subset \supp(CA^t x)$ for large enough $t$. But since we have $\supp(CA^t x) \subset \supp(Cx)\cup\supp(CAx)\cup\dots\cup\supp(CA^{n-1}x)$ by the Cayley-Hamilton theorem, we have $S \subset \supp(Cx)\cup\supp(CAx)\cup\dots\cup\supp(CA^{n-1}x)$. Finally since this is true for any $x \neq 0$, and since $|S| > 2q$, we conclude by \ref{prop:NSCl0} that $q$ errors are correctable after $n$ steps.
\end{itemize}
\end{proof}

We now use this lemma to prove Proposition \ref{prop:stateFeedback}:

\begin{proof} (Proof of Proposition \ref{prop:stateFeedback})\\
To prove the result, we will show that if the chosen poles $\lambda_1, \dots, \lambda_n$ have distinct magnitudes and do not fall in some finite set $F$, then there is a choice of $K$ such that the eigenvalues of $A+BK$ are exactly the $\lambda_1, \dots, \lambda_n$, and the corresponding eigenvectors $v_i$ are such that $|\supp(Cv_i)|=p$. Thus, by the previous lemma, this will show that the number of correctable errors for $(A+BK,C)$ is $\lceil p/2 - 1 \rceil$.

First note that if $\lambda$ is an eigenvalue of $A+BK$ and $x$ is a corresponding eigenvector, then we have $Ax+BKx=\lambda x$, or, if $(\lambda I - A)^{-1}$ is well defined, $x = (\lambda I - A)^{-1} BKx$, i.e., $x$ is proportional to the vector $(\lambda I - A)^{-1} B$ (since $Kx$ is a real number). This means that if $\lambda$ is an eigenvalue of $A+BK$, then necessarily the corresponding eigenvector is $(\lambda I - A)^{-1} B$.

We will therefore look for values of $\lambda$ for which $C(\lambda I - A)^{-1} B$ has full support.

Let $i \in \{1,\dots,p\}$ be fixed and denote by $e_i$ the vector in $\field{R}^p$ that has a 1 in the $i$th component and zeros elsewhere. Note that since $(A,B)$ is controllable there exists $\lambda$ such that $e_i^T C(\lambda I - A)^{-1} B \neq 0$ (see 
\cite[Chapter 3, Theorem 2.17(ii)]{antsaklis2005linear}), and in fact the set $F_i = \{ \lambda \in \field{C}\suchthat e_i^T C (\lambda I - A)^{-1} B = 0 \} \subseteq \field{C}$ is finite (zeros of a non-identically-zero rational fraction).

Now consider $F =(\cup_{i=1}^n F_i)$, and let $\lambda_1, \dots, \lambda_n$ be any choice of $n$ numbers in $\field{C} \backslash F$ with distinct magnitudes. We will show that there exists $K$ such that the eigenvalues of $A+BK$ are the $\lambda_j's$ and the eigenvectors $v_j$ are such that $Cv_j$ has full support.

By controllability of $(A,B)$ there is a $K$ such that the eigenvalues of $A+BK$ are the $\lambda_j$'s. We know that the eigenvectors of $A+BK$ are the $v_j = (\lambda_j I - A)^{-1} B$. Now by the choice of the $\lambda_j$'s and by the definition of $F$ we know that for all $j$ and for any $i$, $e_i^T C(\lambda_j I - A)^{-1} B \neq 0$. In other words, for any $j$, the vector $Cv_j$ has full support. Hence, by lemma \ref{lem:CorrectionEigenvector}, the number of correctable errors of $(A+BK,C)$ is maximal.
\end{proof}

\subsection{Optimization formulation of the optimal decoder}
\label{sec:decodingOpt}

In the previous sections we have discussed and quantified the resilience of a given system $(A,C)$ by characterizing the maximum number of attacks that are tolerable so that the initial state of the system could still be exactly recovered. We saw that if the system $(A,C)$ satisfies the condition
\begin{equation}
\label{eq:CorrQerrors}
 |\supp(Cz)\cup\dots\cup\supp(CA^{T-1} z)| > 2q, \quad \forall z \neq 0
\end{equation}
then it is possible to correct any attacks on $q$ sensors using the $T$ observations $y^{(0)}, \dots, y^{(T-1)}$. We did not discuss however how to actually \emph{recover} the state $x^{(0)}$ from the observations. In this section we focus on the problem of \emph{constructing} a decoder that can correct any number $q$ of errors as long as $q$ satisfies the condition \eqref{eq:CorrQerrors} above.

Consider the decoder $D^T_0:(\field{R}^p)^T \rightarrow \field{R}^n$
defined such that $D^T_0(y^{(0)}, \dots, y^{(T-1)})$ is the optimal
$\hat{x}$ solution of the following optimization problem:
\begin{equation}
\begin{array}{cl}
\underset{\hat{x} \in \field{R}^n, \hat{K} \subset \{1,\dots,p\}}{\text{minimize}} & \;|K| \\
\text{subject to} & \; \supp(y^{(t)} - C A^t \hat{x}) \subset \hat{K} \text{ for } t\in\{0,\dots,T-1\}.
\end{array}
\label{eq:l0dec}
\end{equation}
Observe that the decoder $D_0^T$ looks for the smallest set $K$ of
attacked sensors that can explain the received data $y^{(0)}, \dots,
y^{(T-1)}$. We show in the next proposition that the decoder $D_0^T$
is optimal in terms of error-correction capabilities.

\begin{prop}
Assume that $q$ errors are correctable after $T$ steps, i.e., that
\eqref{eq:CorrQerrors} holds. Then the decoder $D_0^T$ corrects $q$
errors, i.e., for any $x^{(0)} \in \field{R}^n$, and any $e^{(0)},
\dots, e^{(T-1)}$ in $\field{R}^p$ such that $\supp(e^{(t)}) \subset
K$ with $|K| \leq q$, we have $D_0^T(y^{(0)}, \dots,
y^{(T-1)})=x^{(0)}$ where $y^{(t)} = CA^t x^{(0)} + e^{(t)}$.
\end{prop}
\begin{proof}
  Let $x^{(0)}$ and the $e^{(t)}$'s satisfy the stated assumptions,
  with $\supp(e^{(t)}) \subset K$ and $y^{(t)} = CA^t + e^{(t)}$.
  Assume for the sake of contradiction that the feasible point
  $(x^{(0)},K)$ is not the unique optimal point for
  \eqref{eq:l0dec}. Hence there exists $x_a \neq x^{(0)}$, and
  $e^{(0)}_a, \dots, e^{(T-1)}_a$ with $\supp(e^{(t)}_a) \subset K_a$
  that generate the same sequence $y^{(0)}, \dots, y^{(T-1)}$ of
  observed values, with in addition, $|K_a| \leq |K| \leq q$. We
  therefore have two different initial conditions $x^{(0)} \neq x_a$
  and two different error vectors corresponding to less than $q$
  attacked sensors that generate exactly the same sequence of observed
  values. This exactly means that $q$ errors are \emph{not}
  correctable after $T$ steps which contradicts the assumption.
\end{proof}

The proposition above therefore shows that the decoder $D_0^T$ is the best decoder
in terms of error-correction capabilities, since if any decoder can correct $q$ errors, then $D_0^T$ can
as well. One issue however is that the optimization problem
\eqref{eq:l0dec} is not practical since it is NP-hard in
general. Indeed for the special case $T=1$ (corresponding to the case
of ``static'' error-correction over the reals mentioned earlier) the
decoder becomes
\begin{equation}
\underset{x \in \field{R}^n}{\text{minimize}} \quad \|y - Cx\|_{\ell_0}
\label{eq:errcorrl0}
\end{equation}
(where $\|z\|_{\ell_0} = |\supp(z)|$) which is known to be NP-hard
(see for example \cite{guruswami2008euclidean}).

However, in \cite{candes2005decoding}, Candes and Tao proposed to
replace the $\ell_0$ ``norm'' by an $\ell_1$ norm, thereby
transforming the problem into a convex program that can be efficiently
solved:
\begin{equation*}
\underset{x \in \field{R}^n}{\text{minimize}} \quad \|y - Cx\|_{\ell_1}.
\end{equation*}
It was then shown in \cite{candes2005decoding} that if the matrix $C$
satisfies certain conditions, then the solution of this convex program
is the same as the one given by the $\ell_0$ optimal decoder.  In the
next section we consider this transformation in the context of our
problem.

\subsection{The $\ell_1$ decoder: a relaxation of the optimal decoder}
\label{sec:ell1decoder}

For $T \in \field{N}\backslash\{0\}$, consider the linear map $\Phi^{(T)}$ defined by:
\begin{align*}
\Phi^{(T)} \colon & \field{R}^n \rightarrow \field{R}^{p \times T}\\
                    & x \mapsto \begin{bmatrix} Cx & | & CAx & | & \dots & | & CA^{T-1} x \end{bmatrix}.
\end{align*}
Furthermore, if $y^{(0)}, \dots, y^{(T-1)} \in \field{R}^p$, let $Y^{(T)}$ the $p \times T$ matrix formed by concatenating the $y^{(t)}$'s in columns:
\begin{equation*}
Y^{(T)} = \begin{bmatrix} y^{(0)} & | & y^{(1)} & | & \dots & | & y^{(T-1)} \end{bmatrix} \in \field{R}^{p \times T}.
\end{equation*}
Recall that for a matrix $M \in \field{R}^{p \times T}$ with rows
$M_1, \dots, M_p \in \field{R}^T$ the $\ell_0$ ``norm'' of $M$ is the
number of nonzero rows in $M$:
\[ \|M\|_{\ell_0} = |\rowsupp(M)| = |\{ i \in \{1,\dots,p\} \suchthat M_i \neq 0 \}|. \]
Observe that the optimal decoder $D_0^T$ introduced in the previous section can be written as:
\[ D_0^T(y^{(0)}, \dots, y^{(T-1)}) = \argmin_{x \in \field{R}^n} \|Y^{(T)} - \Phi^{(T)} x\|_{\ell_0}. \]
As we saw in the previous section, this decoder finds the minimum number of attacked sensors that can explain the received data $y^{(0)}, \dots, y^{(T-1)}$.

Analogously to \cite{candes2005decoding}, we can define an $\ell_1$
decoder which, instead of minimizing the number of nonzero rows, minimizes the sum of the magnitudes of each row. Specifically, if we measure the magnitude of a row by its $\ell_r$
norm in $\field{R}^T$ (for $r \geq 1$), we obtain the following
decoder $D_{1,r}^T$:
\begin{equation}
 D_{1,r}^T(y^{(0)}, \dots, y^{(T-1)}) = \argmin_{x \in \field{R}^n}
 \|Y^{(T)} - \Phi^{(T)} x\|_{\ell_1/\ell_r}
 \label{eq:optl1}
\end{equation}
where, by definition, $\|M\|_{\ell_1/\ell_r}$ is the sum of the $\ell_r$ norms of the rows of the matrix $M$:
\[ \|M\|_{\ell_1/\ell_r} = \sum_{i=1}^p \|M_i\|_{\ell_r}. \]
Note that the optimization problem in \eqref{eq:optl1} is convex and
can be efficiently solved. Also note that such ``mixed'' $\ell_1/\ell_r$ norms were also used in the compressed sensing literature in the context of joint-sparse and block-sparse signal recovery \cite{eldar2009block}.

We saw in Proposition \ref{prop:NSCl0} that the number of errors that
can be corrected by the optimal $\ell_0$ decoder $D_0^T$ is equal to
the largest number $q$ such that
$|\supp(Cz)\cup\supp(CAz)\cup\dots\cup\supp(CA^{T-1}z)| > 2q$ for all
$z \neq 0$.

The next proposition characterizes the maximum number of errors that
can be corrected by the $\ell_1/\ell_r$ decoder $D_{1,r}^T$.
\begin{prop}
The following are equivalent:\\ (i) The decoder $D_{1,r}^T$ can
correct $q$ errors after $T$ steps.\\ (ii) For all $K \subset
\{1,\dots,p\}$ with $|K|=q$ and for all $G=\Phi^{(T)} z$ with $z \in
\field{R}^n \backslash \{0\}$ we have:
\begin{equation}
 \sum_{i \in K} \|G_i\|_{\ell_r} < \sum_{i \in K^c} \|G_i\|_{\ell_r}.
 \label{eq:NSCl1}
\end{equation}
\label{prop:NSCl1}
\end{prop}
\begin{proof}
\textbf{(i) $\Rightarrow$ (ii)}: Suppose for the sake of contradiction
that (ii) does not hold. Then there exists $K \subset \{1,\dots,p\}$
with $|K|=q$, and $G=\Phi^{(T)} z \in \field{R}^{p \times T}$ with $z \neq 0$ such that $\sum_{i
  \in K} \|G_i\|_{\ell_r} \geq \sum_{i \in K^c} \|G_i\|_{\ell_r}$. Let
$x^0 = 0$ and define the $K$-supported error vectors $e^{(t)}$, for $t \in
\{0,\dots,T-1\}$ by $e^{(t)}_i = G_{i,t}$ if $i \in K$ and $e^{(t)}_i = 0$ otherwise.
Now consider $y^{(t)} = CA^t x^0+e^{(t)}=e^{(t)}$ and let $Y^{(T)}$ be,
as before, the $p \times T$ matrix obtained by concatenating the
$y^{(t)}$'s in columns. Note that $\rowsupp(Y^{(T)}) = K$, and that
$Y^{(T)}_i = (\Phi^{(T)}z)_i$ for all $i \in K$.  We
  will now show that the objective function for \eqref{eq:optl1} at
  $z\neq 0$ is smaller than at $x^0=0$, which will show that the decoder $D_{1,r}^T$ fails to reconstruct $x^{(0)}$ from the $y^{(t)}$'s. This will show that (i) is not
  true. Indeed we have:
\begin{align*}
\|Y^{(T)} - \Phi^{(T)} z\|_{\ell_1/\ell_r} &= \sum_{i=1}^n \|(Y^{(T)} - \Phi^{(T)} z)_i\|_{\ell_r} = \sum_{i \in K^c} \|G_i\|_{\ell_r}  \\
																						& \leq \sum_{i \in K} \|G_i\|_{\ell_r} = \sum_{i=1}^n \|(Y^{(T)} - \Phi^{(T)} x^0)_i\|_{\ell_r} = \|Y^{(T)} - \Phi^{(T)} x^0\|_{\ell_1/\ell_r}.
\end{align*}
\textbf{(ii) $\Rightarrow$ (i)}: We again resort to
contradiction. Suppose that (i) is not true. This means there exists
$x^{(0)}$, and $e^{(0)}, \dots, e^{(T-1)}$ with $\supp(e^{(t)})
\subset K$ with $|K| \leq q$ such that $D_{1,r}^T(y^{(0)}, \dots,
y^{(T-1)}) \neq x^{(0)}$ where $y^{(t)} = CA^t x^{(0)} + e^{(t)}$
(i.e., the decoder $D_{1,r}^T$ fails to reconstruct $x^{(0)}$ from the
$y{(t)}$'s).  By definition of the decoder $D_{1,r}^T$, this means
that there exists $\tilde{x} \neq x^{(0)}$ that achieves a smaller
$\ell_1/\ell_r$ objective than $x^{(0)}$:
\[ \sum_{i=1}^n \|(Y^{(T)} - \Phi^{(T)} \tilde{x})_i\|_{\ell_r} \leq \sum_{i=1}^n \|(Y^{(T)} - \Phi^{(T)} x^{(0)})_i\|_{\ell_r}. \]
Now let $z = \tilde{x} - x^{(0)} \neq 0$, and let $G = \Phi^{(T)} z = U-V$ with $U = Y^{(T)} - \Phi^{(T)} x^{(0)}$ and $V = Y^{(T)} - \Phi^{(T)} \tilde{x}$. We have
\begin{equation*}
\sum_{i \in K} \|G_i\|_{\ell_r} = \sum_{i \in K} \|U_i - V_i\|_{\ell_r} \geq  \sum_{i \in K} \|U_i\|_{\ell_r} - \|V_i\|_{\ell_r}
\end{equation*}
Now since $\rowsupp(U) \subset K$, and since $\tilde{x}$ achieves a smaller $\ell_1/\ell_r$ objective than $x^0$, we have $\sum_{i \in K} \|U_i\|_{\ell_r} = \sum_{i=1}^n \|U_i\|_{\ell_r} \geq \sum_{i=1}^n \|V_i\|_{\ell_r}$. Hence we have
\begin{equation*}
\sum_{i \in K} \|G_i\|_{\ell_r} \geq \sum_{i=1}^n \|V_i\|_{\ell_r} - \sum_{i \in K} \|V_i\|_{\ell_r} = \sum_{i \in K^c} \|V_i\|_{\ell_r} =
\sum_{i \in K^c} \|G_i\|_{\ell_r}
\end{equation*}
where the last equality is because $\rowsupp(U) \subset K$.
Hence (ii) is not true.
\end{proof}

Observe that, as expected, if the $\ell_1/\ell_r$ decoder can correct
$q$ errors, then the $\ell_0$ decoder can correct $q$ errors as
well. Indeed, if we assume the opposite, then by Proposition
\ref{prop:NSCl0} there exists $z \neq 0$ such that
$|\supp(Cz)\cup\dots\cup\supp(CA^{T-1}z)| \leq 2q$, which is
equivalent to saying that $|\rowsupp(\Phi^{(T)}z)| \leq 2q$. Now let
$G=\Phi^{(T)}z$ and let $K$ be the $q$ rows of $G$ with the largest
$\ell_r$ norms, then we clearly have $\sum_{i \in K} \|G_i\|_{\ell_r}
\geq \sum_{i \in K^c} \|G_i\|_{\ell_r}$, which contradicts the
condition of the previous proposition.

As a matter of fact, the condition of the previous proposition (for
the $\ell_1/\ell_r$ decoder) is in some sense a more quantitative
version of the condition of proposition \ref{prop:NSCl0} for the
$\ell_0$ decoder. The two conditions guarantee that the row components
of $\Phi^{(T)} z$ are sufficiently spread and are not too concentrated
on a small subset of the rows.

As an illustration, consider the simple example where the number of
sensors is $p=n$ and $C=I_n$ (i.e., we have one
sensor per component of the state $x \in \field{R}^n$) and where
$A$ is the cyclic permutation given by:
\begin{equation}
 A = \begin{bmatrix}
0 & 1 & \dots & 0 & 0\\
\vdots & \vdots &   & \ddots & 0\\
0 & 0 & \dots & 0 & 1\\
1 & 0 & 0 & \dots & 0
\end{bmatrix}
\label{eq:circularPermutation}
\end{equation}
It is easy to see that after $T=n$, the rows of the matrix $\Phi^{(n)} z = \begin{bmatrix} z & Az & \dots & A^{n-1} z \end{bmatrix}$
are identical up to a permutation, and so the $\ell_r$ norm of any two
rows of $\Phi^{(T)} z$ are equal. This shows that for any subset $K$
of rows with $|K| < n/2$, we have $\sum_{i \in K}
\|(\Phi^{(n)}z)_i\|_{\ell_r} < \sum_{i \in K^c}
\|(\Phi^{(n)}z)_i\|_{\ell_r}$, which shows that the $\ell_1/\ell_r$
decoder can correct a maximal number of errors after $n$ steps,
namely, $\lceil n/2 - 1 \rceil$.

%

Finally note that the condition of Proposition \ref{prop:NSCl1} for
the $\ell_1/\ell_r$ decoder corresponds to the well-known ``nullspace
property'' in compressed sensing and sparse signal recovery
\cite{davenport2011introduction}. 


\subsection{Numerical simulations}
\label{sec:simulations}	

In this section we show the performance of the proposed decoding
algorithm first on a random toy example and then on a more realistic system
modeling an electric power network.

\subsubsection{Random system}
\label{sec:simulations_random_sensor_attacks}
We first consider the $\ell_1/\ell_2$ decoder on a system of size
$n=25$, $p=20$ where $A \in \field{R}^{25 \times 25}$ and $C \in
\field{R}^{20 \times 25}$ have iid Gaussian entries. For different
values $q$ of attacked sensors, we tested the decoder on 200 different
initial conditions $x^{(0)}$ and attacked sensors $K \subset \{1,
\dots, p\}$ with $|K|=q$. The initial conditions $x^{(0)}$ were
randomly generated from the standard Gaussian distribution, and the
attack sets were chosen uniformly at random from the set of subsets of
$\{1,\dots,p\}$ of size $q$. Figure
\ref{fig:simulations_random_sensor_attacks_a} shows the fraction of
initial conditions that were correctly recovered by the
$\ell_1/\ell_2$ decoder (cf. equation \eqref{eq:optl1}) in less than
$T = 15$ time steps for the different values of $q$. We see that for
$q$ less than 6 all the initial conditions were correctly recovered in
less than $T=15$ time steps. Figure
\ref{fig:simulations_random_sensor_attacks_b} shows the number of time
steps that it took in average to correctly recover the initial state,
as a function of the number of corrupted components $q$. We see that
as $q$ increases, more measurements were needed to correctly recover
the state of the system.

For each simulation, the attack values (i.e., the values injected by
the attacker in the components $K$) were chosen randomly from a
Gaussian distribution. In order to illustrate the fact that the
decoder can handle unbounded attacks, the magnitude of the attacks
were chosen to be 20 times larger than the magnitude of the
state. Furthermore, the matrix $A$ was appropriately scaled so it has
a spectral radius of 1. The optimization problems were solved using
\texttt{CVX} \cite{cvx}.

\begin{figure}[htbp]
\centering
\subfloat[]{\includegraphics[width=6.5cm]{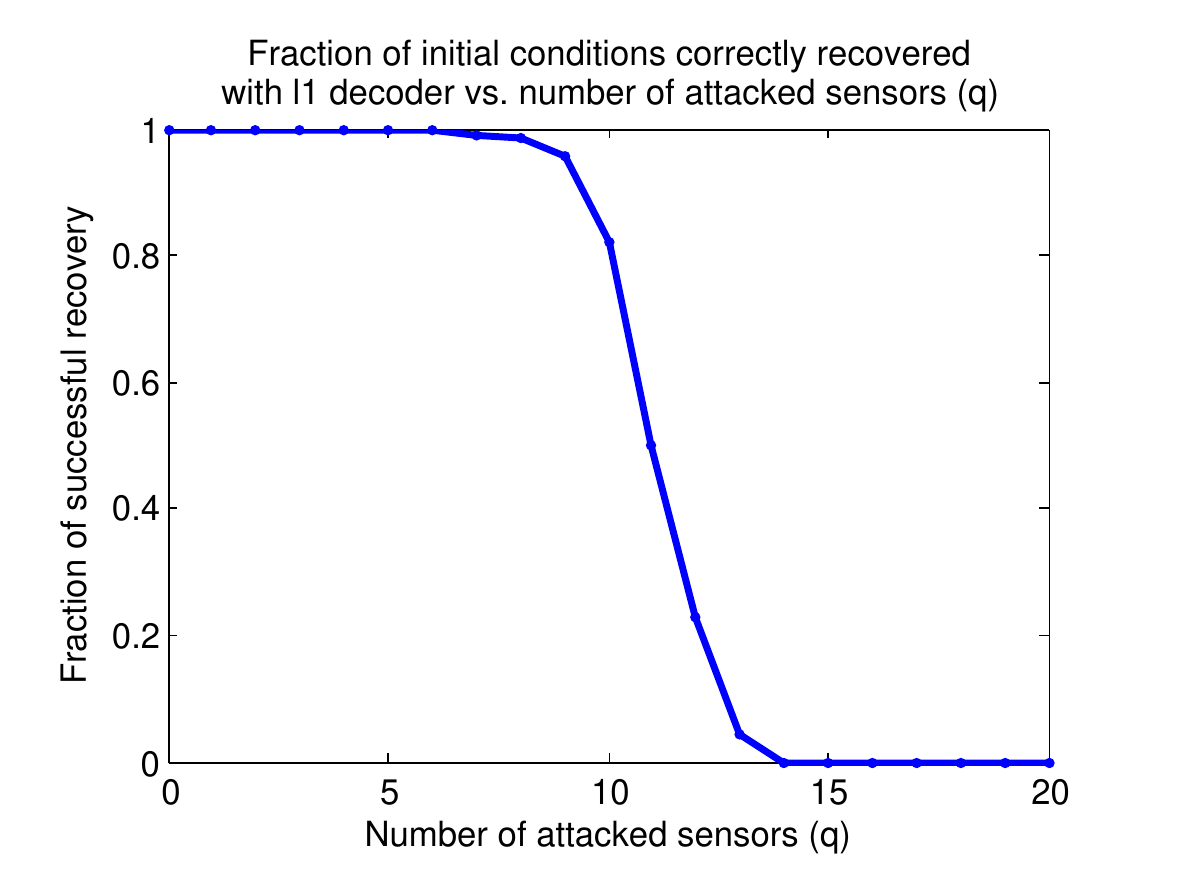}\label{fig:simulations_random_sensor_attacks_a}}\hspace{1.5cm}
\subfloat[]{\includegraphics[width=6.5cm]{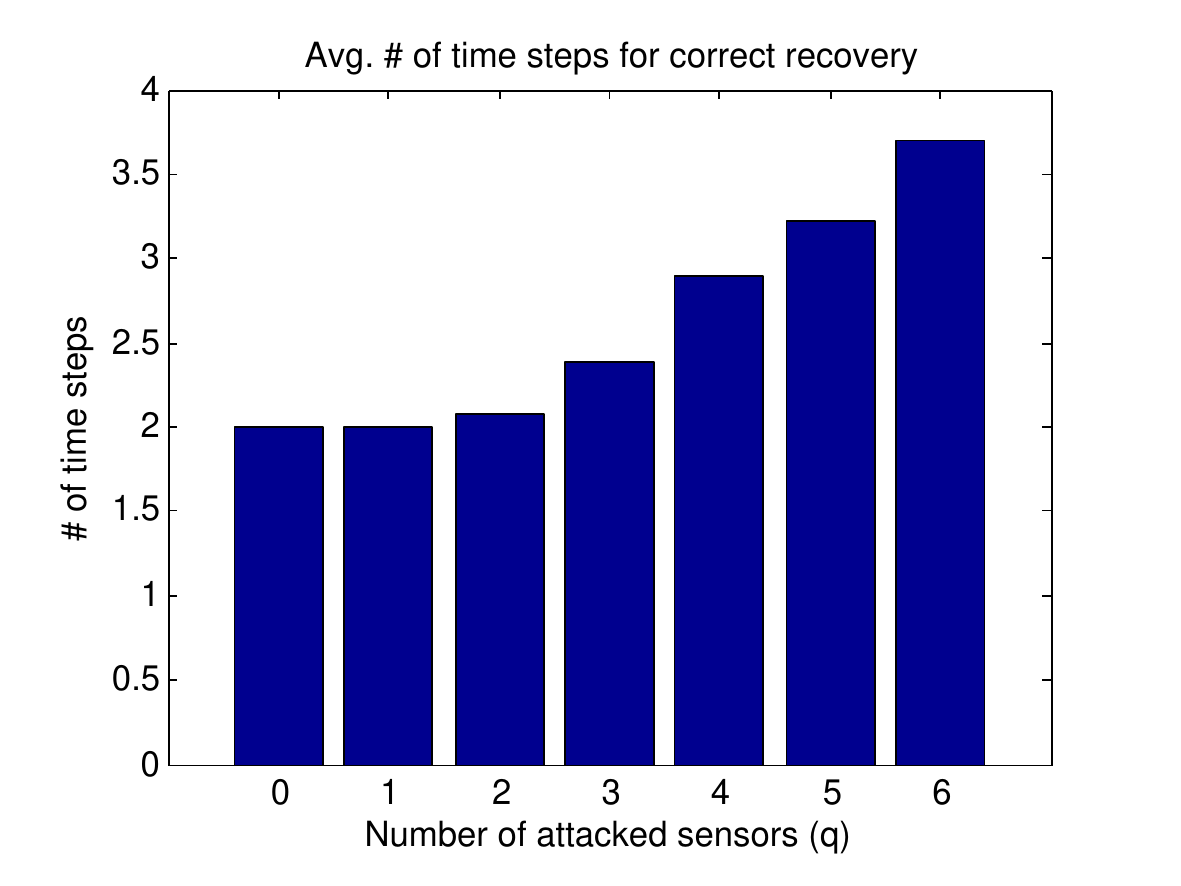}\label{fig:simulations_random_sensor_attacks_b}}
\caption{Performance of $\ell_1/\ell_2$ decoder on a randomly chosen
  system with $n=25$ states and $p=20$ sensors. We see that when the
  number of attacked sensors is small enough, the $\ell_1/\ell_2$ can
  still recover the exact state of the system (left). As the number of
  attacked sensors increases, more measurements are needed to
  correctly recover the state of the system (right).}
\label{fig:simulations_random}
\end{figure}
\subsubsection{Electric power network}
\label{sec:simulations_power_network_sensor_attacks}

In this section we apply the proposed decoding algorithm on a model of
an electric power network and more specifically on the IEEE 14-bus
power network \cite{christie2000power}. The network, depicted in
Figure \ref{fig:ieee14bus} is composed of 5 synchronous generators and
a total of 14 buses. The system is represented by $2\times 5 = 10$
states giving the rotor angles $\delta_i$ and the frequencies
$\omega_i = d\delta_i/dt$ of each generator. Under some simplifying
assumptions the evolution of the system can be captured by a linear
difference equation corresponding to the linearized swing equations
(see \cite{pasqualetti2011graph} for the derivation of the
equations). We assume, like in \cite{pasqualetti2011cyber}, that
$p=35$ sensors are deployed and measure at every time step the real
power injections at every bus (14 sensors), the real power flows along
every branch (20 sensors), and the rotor angle at generator 1 (1
sensor).

For different values of attacked sensors $q$, we ran 200 simulations
with different sets of attacked sensors $K$ of cardinality $q$, and
different initial conditions $x^{(0)}$ that were randomly generated
like in the previous example. In the simulations we did not allow the
last sensor measuring the rotor angle to be attacked. Indeed if this
sensor is attacked then there is no hope of correctly recovering the
state since the system becomes unobservable. The $\ell_1$ decoder we
used however is the one described in Section
  \ref{sec:ell1decoder} and did not incorporate the knowledge of the
  unattacked sensor in any way. Figure
\ref{fig:simulations_power_network_sensor_attacks} shows the number of
simulations (out of the 200) where the state $x^{(0)}$ was correctly
recovered using the $\ell_1/\ell_{\infty}$ decoder in less than $T =
10$ steps. Observe that for $q \leq 4$ the success rate of the decoder
was 100\%. Furthermore when $q \leq 12$ the decoder correctly recovers
the state in more than 90\% of the cases. These simulations show that
the $\ell_1/\ell_r$ decoder works very well in this example and
therefore is a promising practical technique.

\begin{figure}[htbp]
  \centering
  \subfloat[]{\includegraphics[width=6.8cm]{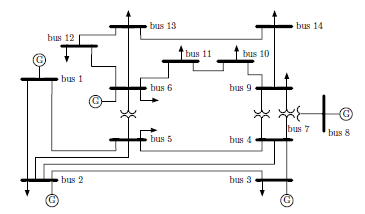}\label{fig:ieee14bus}}\hspace{1.5cm}
  \subfloat[]{\includegraphics[width=6.5cm]{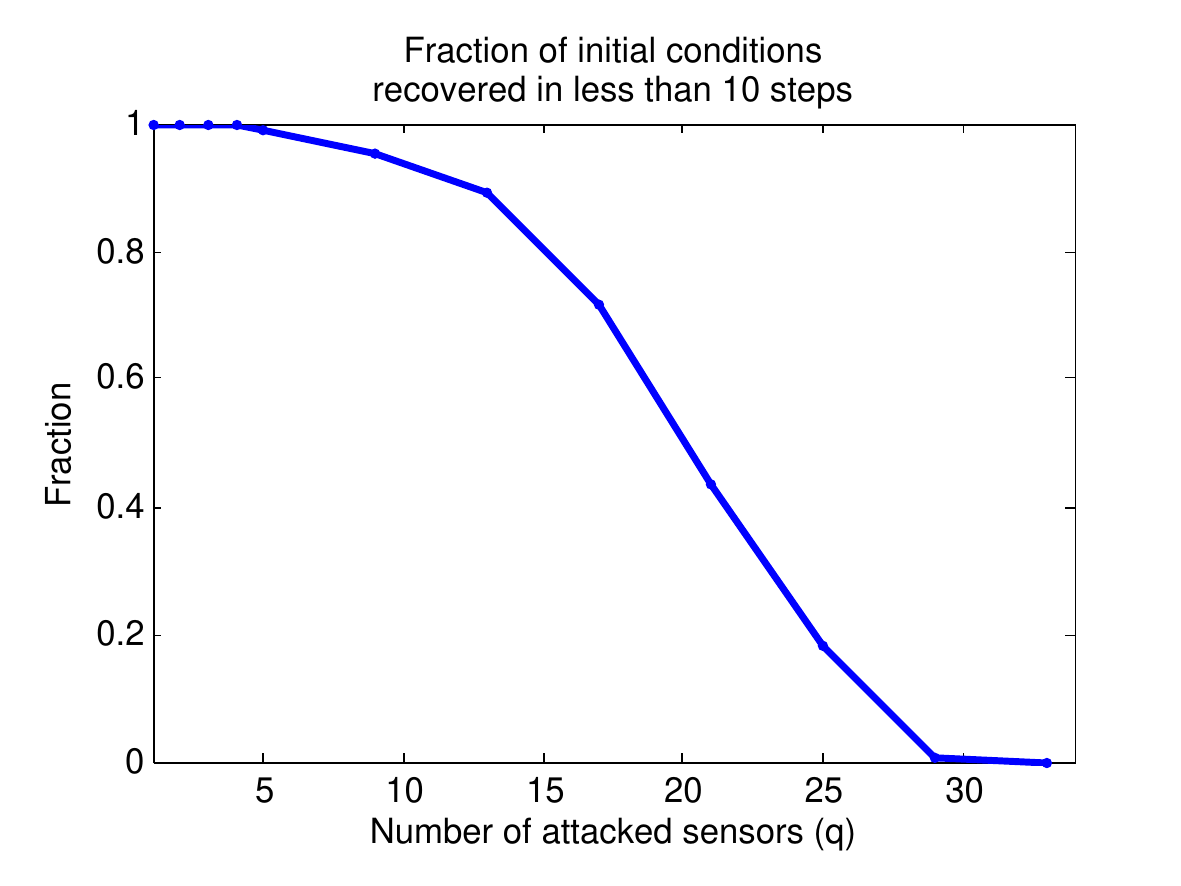}\label{fig:simulations_power_network_sensor_attacks}}
  \caption{Performance of the $\ell_1/\ell_\infty$ decoder on the IEEE
    14-bus power network example. The decoder successfully recovered
    the state of the system when the number of attacked sensors was
    less than 4. When the number of attacked sensors was less than 12,
    the decoder recovered the correct state in more than 90\% of the
    cases.}
\end{figure}


\subsection{The case of attacks on actuators}
\label{sec:Estimation_ActuatorAttacks}

In this section we incorporate into our model attacks on actuators (in addition to attacks on sensors) and we study the resilience of linear control systems to such attacks.
Consider a plant that evolves according to the equations:
\begin{equation}
\begin{aligned}
x^{(t+1)} &= Ax^{(t)}+B(K^{(t)}(y^{(0)}, \dots, y^{(t)})+w^{(t)})\\
y^{(t)}   &= Cx^{(t)} + e^{(t)}
\end{aligned}
\label{eq:control_system_attacks_actuators}
\end{equation}
where $A \in \RR^{n \times n}, B \in \RR^{n\times m}, C \in \RR^{p\times n}$ and $(K^{(t)})_{t=0,1,\dots}$ is an output-feedback control law. As before the vectors $e^{(t)}$ represent attacks on sensors. The vectors $w^{(t)}$ represent attacks on actuators: if actuator $j \in \{1,\dots,m\}$ is not attacked, then $w_j^{(t)} = 0$, otherwise actuator $j$ is attacked and $w_j^{(t)}$ can be arbitrary. The set of attacked actuators will typically be denoted by $L$. In this section we will use the letter $q$ to denote the total number of attacked nodes (sensors and actuators), $q = |K|+|L|$.

Our objective is to monitor the state of the plant from the
observations $y^{(t)}$. More formally if $T$ is some time horizon, we
wish to reconstruct the sequence\footnote{Observe that in the previous
  section where we dealt with attacks on sensors only, the objective
  was to only reconstruct the initial state $x^{(0)}$ since we could
  then simply use the dynamics to propagate the initial condition and
  obtain $x^{(1)}, x^{(2)}, \dots,x^{(T-1)}$. When considering attacks
  on actuators, using the dynamics to propagate the initial state
  requires the knowledge of the attacks. Hence in this section we
  explicitly ask for the recovery of the \emph{whole} sequence
  $x^{(0)}, \dots, x^{(T-1)}$. Note also that one could introduce a
  delay parameter $d \in \NN$ and ask for the recovery of the states
  up to time $T-1-d$ only. The results in this section can easily be
  extended to this case, however for ease of exposition we consider
  only the problem of recovering the whole sequence of states up to
  the current time $T-1$.} of states $x^{(0)}, \dots, x^{(T-1)}$ from
the observations $y^{(0)},\dots, y^{(T-1)}$. Observe that
reconstructing the sequence $x^{(0)}, \dots, x^{(T-1)}$ is equivalent
to reconstructing the initial condition $x^{(0)}$ and the vectors
$Bw^{(0)}, \dots, Bw^{(T-2)}$. Now this reconstruction is possible if,
and only if, the map that sends the tuple $(x^{(0)}, Bw^{(0)}, \dots,
Bw^{(T-2)}, e^{(0)}, \dots, e^{(T-1)})$ to the corresponding outputs
$(y^{(0)}, \dots, y^{(T-1)})$ is injective. Using the
notation
\begin{equation*}
\mathcal O_T = \begin{bmatrix} C\\ CA\\ \vdots\\ CA^{T-1} \end{bmatrix} \in \RR^{pT \times n}, \quad
\mathcal M_T = \begin{bmatrix} 
0 & \dots & \dots & 0\\
C & 0     & \dots & 0\\
CA & C    & \dots & 0\\
\vdots & \vdots & \ddots & \vdots\\
CA^{T-2} &  CA^{T-3} & \dots & C
\end{bmatrix} \in \RR^{pT \times n(T-1)}
\end{equation*}
this map is given by:
\begin{equation*}
(x^{(0)}, Bw^{(0)}, \dots, Bw^{(T-2)}, e^{(0)}, \dots, e^{(T-1)})
 \mapsto \mathcal O_{T} x^{(0)} + \mathcal M_{T} \begin{bmatrix} Bw^{(0)} \\ \vdots \\ Bw^{(T-2)} \end{bmatrix} + \begin{bmatrix} e^{(0)} \\ \vdots \\ e^{(T-1)} \end{bmatrix}
\end{equation*}

If this map is injective when the $w^{(t)}$'s and $e^{(t)}$'s are restricted to have less than $q$ nonzero components combined (i.e., $|K|+|L|\leq q$), we say that $q$ attacks are correctable, or equivalently, that the system is resilient against $q$ attacks. More formally we have the following definition:

\begin{defn}
\label{defn:ControlSystemResilient}
Let a control system of the form \eqref{eq:control_system_attacks_actuators} be given.
We say that $q$ attacks are correctable after $T$ steps (or equivalently, that the system is resilient against $q$ attacks after $T$ steps) if there exists a decoder $D:(\RR^p)^T\rightarrow (\RR^n)^T$ such that for any $x^{(0)} \in \RR^n$, for any $w^{(0)}, \dots, w^{(T-2)}$ with $\supp(w^{(t)}) \subseteq L$ and any $e^{(0)}, \dots, e^{(T-1)}$ with $\supp(e^{(t)}) \subseteq K$ with $|K|+|L| \leq q$ we have $D(y^{(0)},\dots, y^{(T-1)}) = (x^{(0)}, Bw^{(0)}, \dots, Bw^{(T-2)})$.
\end{defn}

The previous discussion leads to the following proposition which gives a characterization of the resilience of a linear control system to attacks on sensors and actuators:
\begin{prop}
Let a control system of the form \eqref{eq:control_system_attacks_actuators} be given. The following are equivalent:\\
(i) The system is \emph{not} resilient against $q$ attacks after $T$ steps\\
(ii) There exists $x \neq 0$, and vectors $w^{(0)}, \dots, w^{(T-2)}$ and $e^{(0)}, \dots, e^{(T-1)}$ with $|\supp(w^{(0)})\cup\dots\cup\supp(w^{(T-2)})|+|\supp(e^{(0)})\cup\dots\cup\supp(e^{(T-1)})| \leq 2q$ such that
\begin{equation*}
\mathcal O_{T} x^{(0)} + \mathcal M_{T} \begin{bmatrix} Bw^{(0)} \\ \vdots \\ Bw^{(T-2)} \end{bmatrix} + \begin{bmatrix} e^{(0)} \\ \vdots \\ e^{(T-1)} \end{bmatrix} = 0 \in \RR^{pT}
\end{equation*}
\label{prop:correctionErrorActuators}
\end{prop}

Observe that if a system is resilient against $q$ attacks then necessarily $q < p/2$. Indeed if $q$ attacks (i.e., sensor and actuator attacks) are correctable, then necessarily $q$ sensor attacks are also correctable which implies that $q < p/2$ using the earlier results of section \ref{sec:errorCorrection}. We now show that for most systems the number of correctable errors is maximal and equal to $\lceil p/2 - 1 \rceil$.

\begin{prop} For almost any\footnote{That is, except on a set of Lebesgue measure zero} triple $(A,B,C)$, the number of correctable errors (sensor and actuator errors) after $T = n$ steps is maximal and equal to $\lceil p/2 - 1 \rceil$.
\label{prop:genericResilienceActuators}
\end{prop}
\begin{proof}
Using the notations above for $\mathcal O_T$ and $\mathcal M_T$, consider the matrix
\begin{equation}
 S_{T} = \begin{bmatrix} 
\mathcal O_T & \mathcal M_T & I_{pT}
\end{bmatrix}
\end{equation}
where $I_{pT}$ is the $pT\times pT$ identity matrix. Note that $S_T$ is a $pT \times (n+m(T-1)+pT)$ matrix and its coefficients are all polynomial in the coefficients of $A,B,C$.

Let $K\subseteq \{1,\dots,p\}$ and $L\subseteq \{1,\dots,m\}$ with $|K|+|L|=q$ and consider the following submatrix of $S_T$
\begin{equation}
S_T^{K,L} = \begin{bmatrix} 
\mathcal O_T & \mathcal M^{L}_T & I^{K}_{pT}
\end{bmatrix}
\end{equation}
obtained by keeping the columns indexed by $L$ in each of the $T-1$ column blocks of $\mathcal M_T$ and those indexed by $K$ in each of the $T$ column blocks of $I_{pT}$. Therefore $S_T^{K,L}$ has $pT$ rows and $n+|L|(T-1)+|K|T$ columns. For a given triple $(A,B,C)$, saying that $q$ errors are correctable is equivalent to saying that for any attack pattern $(K,L)$ such that $|K|+|L| \leq 2q$ the map $S_T^{K,L}$ is injective (cf. point (ii) in proposition \ref{prop:correctionErrorActuators}).

We will now show that if $|K|+|L| < p$, then for almost any triple $(A,B,C)$ the matrix $S_T^{K,L}$ is injective. Indeed note first that if $q < p$, then $S_T^{K,L}$ has more rows than columns if the horizon $T$ is large enough and greater than $n$:
\[ 
\begin{aligned}
\text{ncols}(S_T^{K,L}) &= n+|L|(T-1)+|K|T = n-|L|+qT \\
                        &\leq n-|L|+(p-1)T=pT+(n-|L|-T)\leq pT=\text{nrows}(S_T^{K,L})
\end{aligned} \]
where $\text{ncols}$ denotes the number of columns and $\text{nrows}$ denotes the number of rows. Hence $S_T^{K,L}$ is a polynomial matrix in $A,B,C$ that has more rows than columns and thus for almost any choice of $(A,B,C)$, it is injective. \footnote{Indeed observe first that there exists a particular choice of $A,B,C$ such that $S_T^{K,L}$ is injective: take for example $A$ to be the circular permutation matrix [see equation \eqref{eq:circularPermutation}], $B=0$, and $C$ to be the projection on the first $p$ components. Consider now the determinants of the submatrices of $S_T^{K,L}$ which are polynomials in the coefficients of $A,B,C$. Each one of these polynomials is not identically zero --because there exists a particular choice of $A,B,C$ such that $S_T^{K,L}$ is injective-- and so the zero set of each of these polynomials has Lebesgue measure zero, thus the union of these zero sets has Lebesgue measure zero, and so this means that for almost any choice of $A,B,C$ the matrix $S_T^{K,L}$ is injective.} In other words, we showed that the set of $(A,B,C)$ for which $S_T^{K,L}$ is not injective has Lebesgue measure 0.

Therefore, since there are only finitely many attack sets $K$ and $L$, the set of triplets $(A,B,C)$ for which \emph{there exists} $K \subset \{1,\dots,p\}$ and $L \subset \{1,\dots,m\}$ with $|K|+|L| < p$ such that $S_T^{K,L}$ is not injective has Lebesgue measure zero. Hence this means that for almost any triple $(A,B,C)$ the number of correctable errors is maximal (and equal to $\lceil p/2  - 1\rceil$) when $T \geq n$.
\end{proof}

\paragraph{An explicit decoder} We now consider the problem of designing a decoding algorithm that recovers the sequence of states despite attacks on sensors and actuators. We show that one can formulate the decoding problem as an optimization problem, in the same way we did when there were only sensor attacks. Indeed assume we have received measurements $y^{(0)}, \dots, y^{(T-1)}$ and that we wish to reconstruct the sequence of states $x^{(0)}, \dots, x^{(T-1)}$. Then this can be done by solving the following optimization problem:
\begin{equation}
\label{eq:l0dec_actuators}
\begin{array}[t]{ll}
\mbox{minimize}   & |\hat{K}|+|\hat{L}| \\
\mbox{subject to} & \supp(\hat{e}^{(t)}) \subseteq \hat{K}, \supp(\hat{w}^{(t)}) \subseteq \hat{L}\\
                  & y^{(t)} = C\hat{x}^{(t)} + \hat{e}^{(t)}\\
                  & \hat{x}^{(t+1)} = A\hat{x}^{(t)} + B(u^{(t)} + \hat{w}^{(t)})
\end{array}
\end{equation}
The optimization variables are indicated by a ``hat'' (e.g., $\hat{x}^{(t)}$, etc.); the other variables (namely, $y^{(t)}$ and $u^{(t)}$) are given. The optimization program above finds the \emph{simplest} possible explanation of the received data $y^{(0)}, \dots, y^{(T-1)}$, i.e., the one with the smallest number of attacked nodes. One can easily show that if the system is resilient against $q$ attacks (in the sense of definition \ref{defn:ControlSystemResilient}), and if the number of actual attacks is less than $q$, then the output of the optimization problem above gives the correct sequence of states, i.e., $\hat{x}^{(0)}=x^{(0)}, \dots, \hat{x}^{(T-1)} = x^{(T-1)}$.

Unfortunately though, and as we mentioned earlier, solving this optimization problem is hard in general. We can however use the same ideas used previously to relax the decoder by replacing the ``$\ell_0$'' norm (that measures the \emph{cardinality} of the attack set) by an $\ell_1$ norm. When considering attacks on actuators in addition to attacks on sensors this relaxation leads to the following tractable decoder:
\begin{equation}
\label{eq:l1dec_actuators}
\begin{array}[t]{ll}
\mbox{minimize}   & \sum_{i=1}^p \|\hat{E}_i\|_{\ell_r} + \lambda \sum_{i=1}^m \|\hat{W}_i\|_{\ell_r} \\
\mbox{subject to} & \hat{E}_i = (\hat{e}_i^{(0)}, \dots, \hat{e}_i^{(T-1)})\\
				  & \hat{W}_i = (\hat{w}_i^{(0)}, \dots, \hat{w}_i^{(T-2)})\\
                  & y^{(t)} = C\hat{x}^{(t)} + \hat{e}^{(t)}\\
                  & \hat{x}^{(t+1)} = A\hat{x}^{(t)} + B(u^{(t)} + \hat{w}^{(t)})
\end{array}
\end{equation}
For each $i$ the auxiliary variables $\hat{E}_i \in \RR^T$ and
$\hat{W}_i \in \RR^T$ carry the $i$'th components of the attack
vectors over the time horizon $t=0,\dots,T-1$. Thus if
$\|\hat{E}_i\|_{\ell_r} = 0$ then $\hat{e}_i^{(t)} = 0$ for all
$t=0,\dots, T-1$ and the $i$'th sensor is not attacked, and similarly
if $\|\hat{W}_i\|_{\ell_r} = 0$ then the $i$'th actuator is not
attacked. Now observe that the objective function $\sum_{i=1}^p
\|\hat{E}_i\|_{\ell_r} + \lambda \sum_{i=1}^m \|\hat{W}_i\|_{\ell_r}$
is nothing but a weighted sum of the $\ell_1$ norms of the vectors
$(\|\hat{E}_i\|_{\ell_r})_{i=1,\dots,p} \in \RR^p$ and
$(\|\hat{W}_i\|_{\ell_r})_{i=1,\dots,m} \in \RR^m$. Note that we have
introduced a tuning parameter $\lambda$ to control the relative weight
between the term corresponding to the attacks on sensors and the term
corresponding to the attacks on actuators.

\paragraph{Numerical simulations} To illustrate the behavior of the
$\ell_1$ decoder, we tested it on a synthetic randomly-generated system with
$n=15$ states, $m=10$ actuators and $p=10$ sensors\footnote{The system was generated in the same way as the example of section \ref{sec:simulations}: the entries of $A$, $B$, $C$ are iid standard Gaussian, and $A$ was normalized so it has spectral radius 1.}. Figure
\ref{fig:simulations_random_actuator_attacks} shows the performance of
the decoder as a function of the number of attacked sensors and
actuators. We see that on this example the $\ell_1$ decoder correctly
recovers the state of the system despite the attacks when the number
of attacked sensors and actuators is small enough.

Note that the decoder as given in equation \eqref{eq:l1dec_actuators}
depends on the choice of the parameter $\lambda$. For the simulations
of figure \ref{fig:simulations_random_actuator_attacks} we used the
value $\lambda=10$ which we empirically found to be a suitable value
for the system we considered. It would be interesting however to see
if there is a simple and systematic way to directly find the best
value of $\lambda$ from the data and the parameters of the system.

\begin{figure}[htbp]
  \centering
  \subfloat[]{\includegraphics[width=6.5cm]{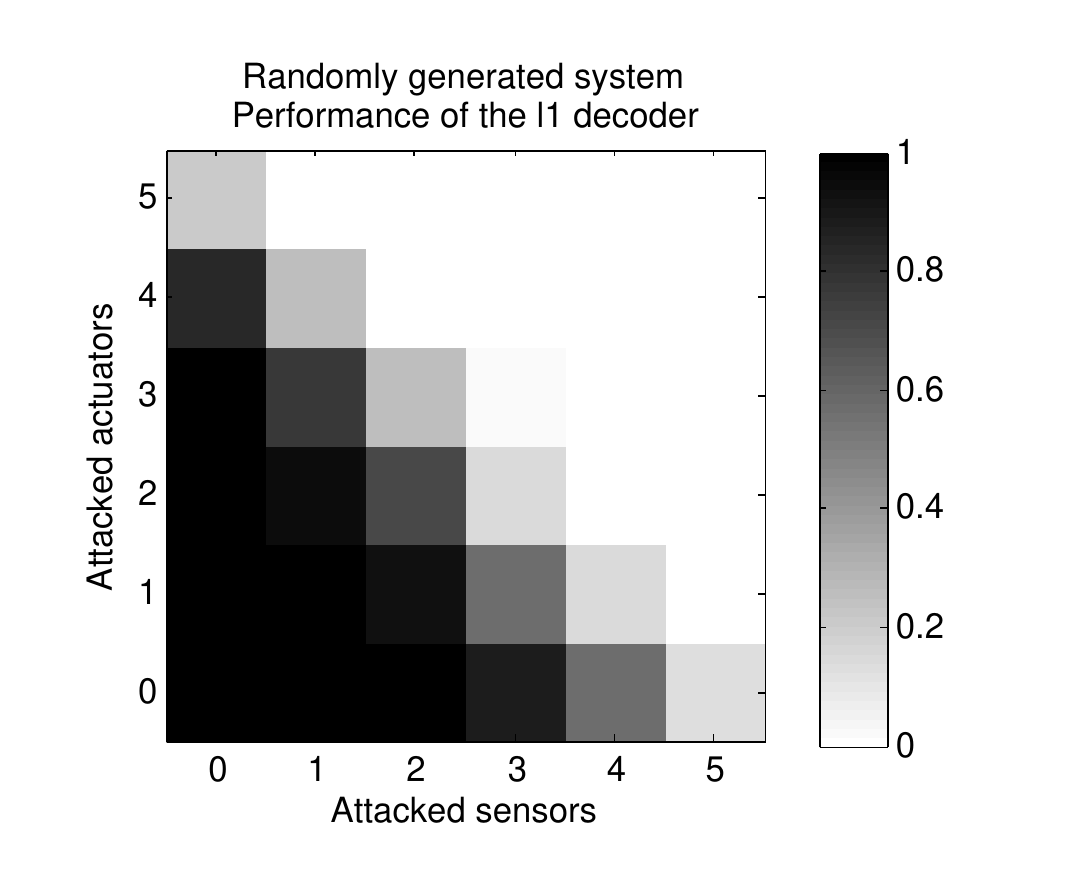}\label{fig:simulations_random_actuator_attacks}}
  \hspace{1.5cm}
 \subfloat[]{\includegraphics[width=6.5cm]{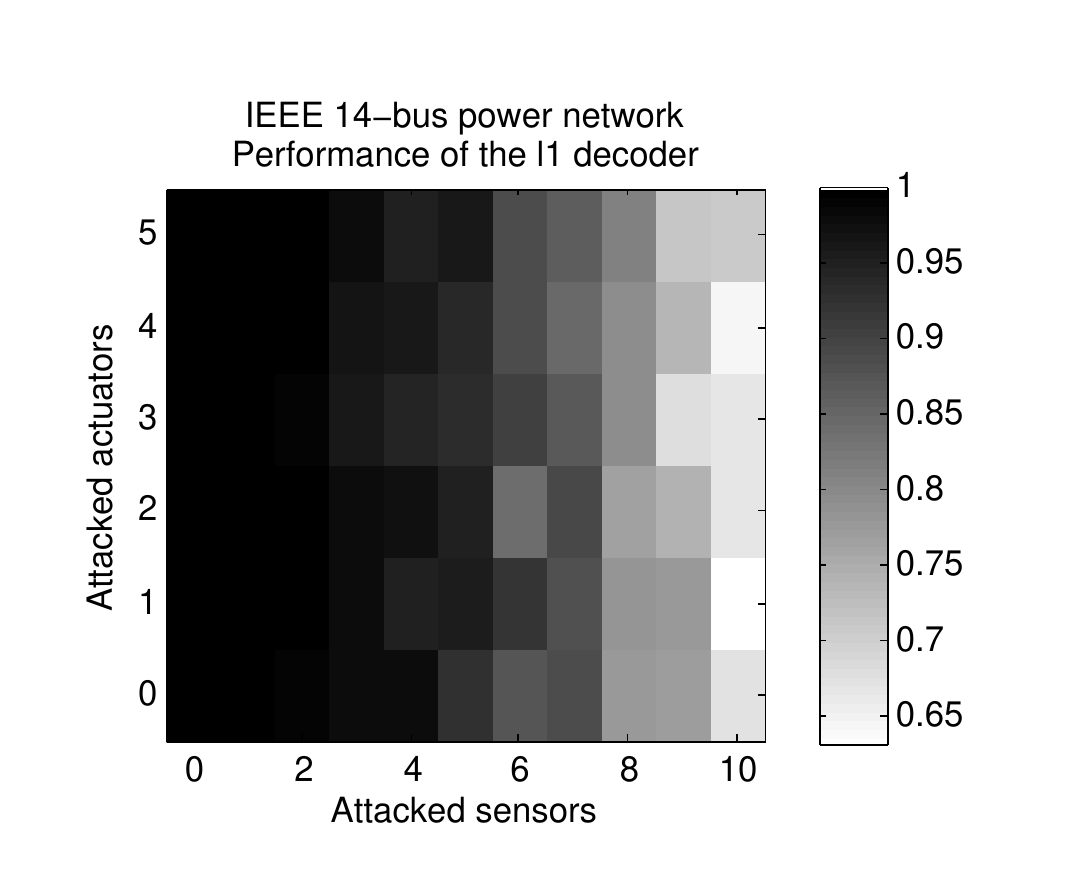}\label{fig:simulations_power_network_actuator_attacks}}
\caption{(Left) Performance of the $\ell_1/\ell_2$ decoder \eqref{eq:l1dec_actuators} (with constant $\lambda=10$) on a randomly generated system with $n=15$ states, $m=10$ actuators and $p=10$ sensors. Dark color indicates a high success rate and white color indicates a low success rate. We observe that when the number of attacked sensors and actuators is small enough the decoder correctly recovers the state of the system. Also we remark that the resilience with respect to attacked sensors decreases as the number of attacked actuators increases, and vice-versa. \newline
(Right) Performance of the $\ell_1/\ell_{\infty}$ decoder \eqref{eq:l1dec_actuators} (with $\lambda=10^{-3}$) on the IEEE 14-bus power network example of section \ref{sec:simulations_power_network_sensor_attacks}. We observe that when the number of attacked sensors is small enough, the decoder correctly recovers the state of the system, independently of the number of attacked generators. This suggests in particular that the system is highly resilient against attacks on generators.}
\end{figure}

We also tested the decoder \eqref{eq:l1dec_actuators} on the power network example of section \ref{sec:simulations_power_network_sensor_attacks}. Recall that the IEEE 14-bus network we considered is comprised of 5 generators and 14 buses, and is modeled by a linear dynamical system with $2\times 5=10$ states (rotor angle $\delta_i$ and frequency $\omega_i = d\delta_i/dt$ for each generator $i$) and 35 sensors. An attack on an actuator here corresponds to an attack on the mechanical power input to a generator $i$ and is modeled by an additive input $w_i$ affecting the equation governing the frequency $\omega_i$ of generator $i$ \cite{pasqualetti2011graph}.

For different values of $|L|$ (number of attacked generators), and $|K|$ (number of attacked sensors), we ran the $\ell_1/\ell_{\infty}$ decoder of equation \eqref{eq:l1dec_actuators} and we recorded its success rate\footnote{We declare the decoder to be successful after $T$ steps if it correctly recovers the whole sequence of states up to time $T-1$ (i.e., with a delay of $d=1$). Indeed, due to the very special structure of the $C$ matrix of the system (the sensors only measure the rotor angles and not the frequencies), it is impossible to reconstruct the state perfectly without delay.} over 200 simulations with different initial conditions and attack sets that were randomly generated like in the previous examples of section \ref{sec:simulations_power_network_sensor_attacks}. The results of the simulations are shown in figure \ref{fig:simulations_power_network_actuator_attacks}. We see that when the number of attacked sensors is small enough, the decoder correctly recovers the state of the system. We also remark that, unlike the previous example of figure \ref{fig:simulations_random_actuator_attacks}, the performance of the decoder is not really affected by the number of attacked generators. This suggests that the system is highly resilient against attacks on the generators since despite these attacks the state of the system can still be correctly recovered from the measurements (when the number of attacked sensors is small).




\section{The control problem with output-feedback}
\label{sec:Control}

In this section we consider general linear control systems with output feedback of the form:
\begin{equation}
\begin{aligned}
x^{(t+1)} &= Ax^{(t)}+BK^{(t)}(y^{(0)}, \dots, y^{(t)})\\
y^{(t)}   &= Cx^{(t)} + e^{(t)}
\end{aligned}
\label{eq:control_system_with_errors_sensors}
\end{equation}


One of the main questions that we address in this section is to
determine whether for a given system $(A,B,C)$, there exists a control
law (i.e., a family $(K^{(t)})_{t=0,1,\dots}$) that drives the state
of the system \eqref{eq:control_system_with_errors_sensors} to the
origin even if some of the sensors are attacked. Observe that the
sensor attacks can affect the control inputs (since the control inputs
are function of the $y^{(t)}$'s) which can in turn deviate the state
$x^{(t)}$ from its nominal path.

Note that if there were attacks on the actuators then such a
stabilizing control law does not exist in general, and that is why we focus only on sensor attacks in this section.


It is clear that if $q$ sensor errors are correctable (in the sense
defined in the previous section, i.e., that it is possible to recover
the state despite any attacks on $q$ sensors), then one can stabilize
the system in the presence of attacks: indeed, one can simply decode
the state (since $q$ errors are correctable), and then apply a
standard \emph{state} feedback law of the form $u=Kx$ (for
example). The main contribution of this section is to show that the
converse of this statement is essentially true. More specifically, we
show in Theorem \ref{thm:SeparationEstimationControl} that if
$(K^{(t)})_{t=0,1,\dots}$ is any feedback law that stabilizes the
system (with a fast enough decay) despite attacks on any $q$ sensors,
then necessarily $q$ errors are correctable. This theorem shows that
one can essentially decouple the problem of estimation and of control
in the scenario we consider: in other words, there is no loss of
resilience in searching for an output feedback law that is the
composition of a decoder with a standard \emph{state} feedback.

\subsection{Some properties}
\label{sec:Control_Properties}

We start by defining the notion of correctability of $q$ errors for systems with output-feedback control inputs. We will see in particular that it is independent of the feedback law used. Recall that the symbol $E_{q,T}$ denotes the set of attack sequences of length $T$ on any $q$ sensors:
\begin{align*}
 E_{q,T} = \Bigl\{ & (e^{(0)}, \dots, e^{(T-1)}) \in (\field{R}^p)^T \suchthat \\
                   & \quad \exists K \subset \{1,\dots,p\}\;, |K|=q\; \forall t \in \{0,\dots,T-1\}, \; \supp(e^{(t)}) \subset K \Bigr\}.
\end{align*}
We also use the notation $y(t,x^{(0)},e)$ to denote the output at time $t$ of the control system \eqref{eq:control_system_with_errors_sensors} when the initial state is $x^{(0)}$ and for the attack sequence $e \in E_{q,T}$.
We now give the definition of correctability of $q$ errors for systems with output-feedback control inputs:
\begin{defn}
Let a control system of the form \eqref{eq:control_system_with_errors_sensors} be given.
We say that $q$ errors are correctable after $T$ steps if there exists a function $D:(\field{R}^p)^T\rightarrow \field{R}^n$ such that for any $x^{(0)} \in \field{R}^n$ and any attack sequence $e \in E_{q,T}$, we have $D\Bigl(y(0,x^{(0)},e), \dots, y(T-1,x^{(0)},e)\Bigr) = x^{(0)}$.
\end{defn}

It is not hard to see that, since the systems we consider are linear and since the control inputs only depend on the measurements, the property of correctability of $q$ errors just defined above does not depend on the control law nor on $B$, and in fact only depends on $A$ and $C$. Indeed, saying that $q$ errors are not correctable (after $T$ steps) for the controlled system $(A,B,C,(K^{(t)})_{t=0,1,\dots})$ means that there exists $x_a \neq x_b$, and error vectors $e_a, e_b \in E_{q,T}$ such that $y(t,x_a,e_a) = y(t,x_b,e_b)$ for all $t=0,\dots,T-1$. In other words, we have, for all $t \in \{0,\dots, T-1\}$:
\begin{equation}
\label{eq:notCorrectableControlSystem}
 CA^t x_a + C [B, AB, \dots, A^{t-1} B] \begin{bmatrix} u_a^{(t-1)} \\ \vdots \\ u_a^{(0)} \end{bmatrix} + e_a^{(t)} = CA^t x_b + C [B, AB, \dots, A^{t-1} B] \begin{bmatrix} u_b^{(t-1)} \\ \vdots \\ u_b^{(0)} \end{bmatrix}  + e_b^{(t)}
\end{equation}
where $u_a^{(\tau)} = U^{(\tau)}(y(0,x_a,e_a),\dots,y(\tau,x_a,e_a))$ and $u_b^{(\tau)} = U^{(\tau)}(y(0,x_b,e_b),\dots,y(\tau,x_b,e_b))$ for $\tau=0,\dots,t-1$.
Now observe that the terms on the left-hand side and right-hand side of \eqref{eq:notCorrectableControlSystem} with the control inputs are equal (since $y(s,x_a,e_a)=y(s,x_b,e_b)$ for all $s$ and thus $u_a^{(\tau)} = u_b^{(\tau)}$). Hence the equality \eqref{eq:notCorrectableControlSystem} is equivalent to saying that for all $t \in \{0,\dots, T-1\}$, we have:
\[ CA^t x_a + e_a^{(t)} = CA^t x_b + e_b^{(t)}. \]
And this exactly means that $q$ errors are not correctable for $(A,C)$. This therefore shows that the notion of correctability does not depend on the control law used.

In other words, one can use the conditions developed earlier for correctability of $q$ errors for linear systems with no inputs and apply them to systems with output-feedback control inputs. For example we have that $q$ errors are correctable for the control system \eqref{eq:control_system_with_errors_sensors} if, and only if, $|\supp(Cz)\cup\dots\cup\supp(CA^{T-1}z)| > 2q$ for all $z \neq 0$.

\subsection{Main result: separation of estimation and control}

We are now ready to state our result on separation of estimation and control.

\begin{thm}
\label{thm:SeparationEstimationControl}
Let $A,B,C$ be three matrices of appropriate sizes and assume that a control strategy given by the $(K^{(t)})_{t=0,1,\dots}$ is such that: \emph{for any} $x^{(0)} \in \field{R}^n$ and \emph{for any} sequence of error vectors $e \in E_{q,T}$, the sequence $(x^{(t)})$ defined by:
\begin{equation}
\begin{aligned}
x^{(t+1)} &= Ax^{(t)}+BK^{(t)}(y^{(0)}, \dots, y^{(t)})\\
y^{(t)}   &= Cx^{(t)} + e^{(t)}
\end{aligned}
\label{eq:control_system}
\end{equation}
satisfies
\begin{equation*}
\|x^{(t)}\| \leq \kappa \alpha^t \|x^{(0)}\|
\end{equation*}
where $\kappa > 0$ and where $0 \leq \alpha < 1$ is small enough: $\alpha < \min\{ |\lambda| \suchthat \lambda \text{ eigenvalue of } A\}$. \emph{Then} necessarily $q$ errors are correctable after $n$ steps.
\end{thm}

\begin{proof}
We proceed by contradiction. Assume that $q$ errors are not correctable after $n$ steps. Then this means there exists a nonzero initial state $\bar{x} \neq 0$ that is indistinguishable from the initial state $0$. In other words, there exist $e_a, e_b \in E_{q,T}$ such that the outputs of the control system \eqref{eq:control_system} in the two different executions:
\begin{enumerate}
\item $x^{(0)} = \bar{x}$ and $e^{(t)} = e_a^{(t)}$; and
\item $x^{(0)} = 0$ and $e^{(t)} = e_b^{(t)}$.
\end{enumerate}
are equal for all $t=0,\dots,n-1$. But by the Cayley-Hamilton theorem, it is not hard to see that the sequences $e_a$ and $e_b$ can be extended to $t \geq n$ so that the outputs of the system \eqref{eq:control_system} are equal for all $t\geq 0$. Observe now that since the control law $K^{(t)}$ only depends on the outputs, this means that in these two executions, the same sequence of inputs, $u^{(t)}$, will be used.

Furthermore, since we must have in both cases, $\|x^{(t)}\| \leq \kappa e^{-\alpha t} \|x^{(0)}\|$, this leads, for the case where $x^{(0)} = 0$, that $x^{(t)} = 0$ for all $t \geq 0$, and so necessarily, $Bu^{(t)} = x^{(t+1)} - Ax^{(t)} = 0$ for all $t \geq 0$.
Hence for the first case (when $x^{(0)} = \bar{x}$), the recurrence relation is $x^{(t+1)}=Ax^{(t)}$, which gives $x^{(t)} = A^t \bar{x}$. We now get a contradiction since $x^{(t)}$ should decay at rate of $\alpha$, but the eigenvalues of $A$ are all strictly larger than $\alpha$. This completes the proof.
\end{proof}

\begin{rem}
Note that the assumption on the decay rate to be fast enough is necessary; otherwise the result is not true. Indeed, if for example $A$ is already a stable matrix, one cannot deduce anything from the mere existence of a stabilizing control law (since the system is by itself stable!). For a concrete example, take $A=0.5 I$, $B=I$, $C=I$ (note that $A$ is stable). We know from the characterization of the number of correctable errors that even 1 error is not correctable after any number of steps (for example if we take $x=(1,0,\dots,0)$, then $|\supp(Cx)\cup\supp(CAx)\cup\dots|=1 \not> 2q$ if $q > 0$). Now if we consider the trivial output feedback law $K^{(t)} = 0$ for all $t$, the resulting system is of course stable despite any number of attacks (the state evolution is simply $x^{(t+1)}=0.5x^{(t)}$ and does not even depend on the sensor outputs), but as we just saw one cannot even construct a decoder to correct even 1 error!
\end{rem}

\section{Conclusion}

In this paper we considered the problems of estimation and control of
linear systems when some of the sensors or actuators are attacked. For
the estimation problem we gave a characterization of the number of
attacks that can be tolerated so that the state of the system can
still be exactly recovered, and we showed how one can increase the
resilience of the system by state-feedback while guaranteeing a
certain performance. We then showed that there is an explicit (though
computationally hard) decoder that can correct the maximal number of
errors. The decoder was then relaxed to obtain a computationally
feasible decoding algorithm which appears to perform well in the
numerical simulations.

We then considered the problem of designing an output feedback law to
stabilize a linear plant where at most $q$ sensors are attacked. Our main result was to show that if such a resilient output-feedback law exists, then necessarily there also exists a decoder that is resilient against $q$ attacks. This shows in particular that there is no loss of resilience in searching for an output-feedback law that is the composition of a decoder with a standard \emph{state} feedback law.


There are many important open questions, which are
  unanswered in this work. For example, the question of constructing
  an iterative estimator (where the estimate of the state is updated by a simple iterative rule each time a new measurement is received), instead of the one-shot $\ell_1$ estimator in this paper, would be interesting in particular from a computational point of view.
  Another subject of interest is to study the effect of exogenous noise on the performance of the estimator presented in this paper.
  Finally, ideas on how to specialize the techniques proposed here to particular
  applications, by taking into account structural vulnerabilities (in
  terms of sets of sensors and actuators attacked) might be of
  interest.

\bibliographystyle{IEEEtran}
\bibliography{IEEEabrv,ControlSecurity}

\end{document}